\documentclass[12pt]{article}

\usepackage[english]{babel}
\usepackage[utf8]{inputenc}
\usepackage[top=2cm, bottom=3cm, left=2.5cm, right=2.5cm]{geometry}
\usepackage{mathrsfs}
\usepackage{amsmath}
\usepackage{amsthm}
\usepackage{amsfonts}
\usepackage{graphicx}
\usepackage{hyperref}

\newcommand{\UIHPT}{\textup{\textsf{UIHPT}}}
\newcommand{\IIC}{\textup{\textsf{IIC}}}
\newcommand{\UIPT}{\textup{\textsf{UIPT}}}
\newcommand{\UIHPQ}{\textup{\textsf{UIHPQ}}}
\renewcommand{\P}{\mathbb{P}}
\newcommand{\m}{\mathbf{m}}
\newcommand{\Pb}{\mathbf{P}}
\newcommand{\tr}{\mathbf{t}}
\newcommand{\lt}{\mathbf{l}}
\newcommand{\M}{\mathcal{M}}
\newcommand{\Ms}{\widehat{\M}}
\newcommand{\dloc}{d_{\textup{loc}}}
\newcommand{\B}{\mathbf{B}}

\newcommand{\A}{\mathsf{A}}
\newcommand{\Ec}{\mathcal{E}}
\newcommand{\Rc}{\mathcal{R}}
\newcommand{\C}{\mathcal{C}}
\newcommand{\T}{\mathcal{T}}
\newcommand{\LT}{\mathcal{L}}
\newcommand{\Tr}{\mathbf{T}}
\newcommand{\Lt}{\mathbf{L}}
\newcommand{\Tinf}{\Tr_\infty}
\newcommand{\Linf}{\Lt_\infty}
\newcommand{\N}{\mathbb{N}}
\newcommand{\Z}{\mathbb{Z}}
\newcommand{\R}{\mathbb{R}}
\newcommand{\mub}{\mu_{\bullet}}
\newcommand{\muw}{\mu_{\circ}}
\newcommand{\GWmm}{\mathsf{GW}_{\muw,\mub}}
\newcommand{\nub}{\nu_{\bullet}}
\newcommand{\nuw}{\nu_{\circ}}
\newcommand{\GWnn}{\mathsf{GW}_{\nuw,\nub}}
\newcommand{\Piic}{\Pb_\IIC}
\newcommand{\Hc}{\mathcal{H}}
\newcommand{\Hb}{\Hc_b}
\newcommand{\Hw}{\Hc_w}
\newcommand{\Hl}{\Hc_l}
\newcommand{\Hr}{\Hc_r}
\newcommand{\Psf}{\mathsf{P}}
\DeclareMathOperator{\Loop}{Loop}
\DeclareMathOperator{\Tree}{Tree}
\DeclareMathOperator{\Treeb}{\textup{\textbf{Tree}}}
\DeclareMathOperator{\Scoop}{Scoop}

\title{\textsc{The Incipient Infinite Cluster of the Uniform Infinite Half-Planar Triangulation}}

\author{Loïc Richier%
\thanks{UMPA, \'Ecole Normale Supérieure de Lyon, 46 allée d'Italie, 69364 Lyon Cedex 07, France. \newline Email: \texttt{loic.richier@ens-lyon.fr}}}

\date{\today}

\newtheorem{Th}{\bf Theorem}
\newtheorem*{Thbis}{\bf Theorem}
\newtheorem{Prop}{\bf Proposition}

\newtheorem{Lem}{\bf Lemma}
\newtheorem{Cor}{\bf Corollary}

\theoremstyle{definition}

\newtheorem{Def}{\bf Definition}

\newtheorem{Alg}{\bf ALGORITHM}
\newtheorem{Rk}{\bf Remark}

\newtheorem*{Ack}{\bf Acknowledgements}

\begin{document}

\maketitle

\begin{abstract}
We introduce the Incipient Infinite Cluster ($\IIC$) in the critical Bernoulli site percolation model on the Uniform Infinite Half-Planar Triangulation ($\UIHPT$), which is the local limit of large random triangulations with a boundary. The $\IIC$ is defined from the $\UIHPT$ by conditioning the open percolation cluster of the origin to be infinite. We prove that the $\IIC$ can be obtained by adding within the $\UIHPT$ an infinite triangulation with a boundary whose distribution is explicit.
\end{abstract}

\section{Introduction}\label{sec:Intro} The purpose of this work is to describe the geometry of a large critical percolation cluster in the (type 2) Uniform Infinite Half-Planar Triangulation ($\UIHPT$ for short), which is the \textit{local limit} of random triangulations with a boundary, upon letting first the volume and then the perimeter tend to infinity. Roughly speaking, rooted planar graphs, or maps, are close in the local sense if they have the same ball of a large radius around the root. The study of local limits of large planar maps goes back to Angel \& Schramm, who introduced in~\cite{angel_uniform_2003} the Uniform Infinite Planar Triangulation ($\UIPT$), while the half-plane model was defined later on by Angel in~\cite{angel_scaling_2004}. Given a planar map, the Bernoulli site percolation model consists in declaring independently every site open with probability $p$ and closed otherwise.

Local limits of large planar maps equipped with a percolation model have been studied extensively. Critical thresholds were provided for the $\UIPT$~\cite{angel_growth_2003} and the $\UIHPT$~\cite{angel_scaling_2004,angel_percolations_2015} as well as for their quadrangular equivalents~\cite{menard_percolation_2014,angel_percolations_2015,richier_universal_2015}. The central idea of these papers is a Markovian exploration of the maps introduced by Angel called the \textit{peeling process}, which turns out to be much simpler in half-plane models: in this setting, various critical exponents~\cite{angel_percolations_2015} and scaling limits of crossing probabilities~\cite{richier_universal_2015} can also be derived.

\medskip

A natural goal in percolation theory is the description of the geometry of percolation clusters at criticality.  In the $\UIPT$, such a description has been achieved by Curien \& Kortchemski in~\cite{curien_percolation_2014}. They identified the \textit{scaling limit} of the boundary of a critical percolation cluster conditioned to be large as a random \textit{stable looptree} with parameter $3/2$, previously introduced in~\cite{curien_random_2014}. Here, our aim is to understand not only the local limit of a percolation cluster conditioned to be large, but also the local limit of the whole $\UIHPT$ under this conditioning. This is inspired by the work of Kesten~\cite{kesten_incipient_1986} in the two-dimensional square lattice.

\smallskip

Precisely, we consider a random map distributed as the $\UIHPT$, equipped with a site percolation model with parameter $p\in[0,1]$, and denote the resulting probability measure by $\Pb_p$ (details are postponed to Section \ref{sec:Definitions}). Angel proved in~\cite{angel_scaling_2004} that the critical threshold $p_c$ equals $1/2$, and that there is no infinite connected component at the critical point almost surely. We also work conditionally on a ``White-Black-White" boundary condition, meaning that all the vertices on the infinite simple boundary of the map are closed, except the origin which is open. We denote by $\C$ the open cluster of the origin, and by $\vert\C \vert$ its number of vertices or \textit{volume}. The exploration of the percolation interface between the origin and its left neighbour on the boundary reveals a closed path in the $\UIHPT$. The maximal length of the \textit{loop-erasure} of this path throughout the exploration is interpreted as the \textit{height} $h(\C)$ of the cluster $\C$. Theorem \ref{thm:IICTheorem} states that 
\[\Pb_p(\cdot \mid \vert \C \vert =\infty) \underset{p \downarrow p_c}{\Longrightarrow}  \Piic\quad \text{and} \quad \Pb_{p_c}(\cdot \mid h(\C) \geq n) \underset{n \rightarrow \infty}{\Longrightarrow}  \Piic\] in the sense of weak convergence, for the local topology. The probability measure $\Piic$ is called (the law of) the Incipient Infinite Cluster of the $\UIHPT$ ($\IIC$ for short) and is supported on triangulations of the half-plane. As in \cite{kesten_incipient_1986}, the limit is universal in the sense that it arises under at least two distinct and natural ways of conditioning $\C$ to be large.

The proof of Theorem \ref{thm:IICTheorem} unveils a decomposition of the $\IIC$ into independent sub-maps with an explicit distribution. We first consider the percolation clusters of the origin and its neighbours on the boundary. By filling in their finite holes, we obtain the associated percolation \textit{hulls}. The boundaries of the percolation hulls are random infinite \textit{looptrees}, that is, a collection of cycles glued along a tree structure introduced in~\cite{curien_random_2014}. The percolation hulls are rebuilt from their boundaries by filling in the cycles with independent \textit{Boltzmann triangulations} with a simple boundary. Finally, the $\IIC$ is recovered by gluing the percolation hulls along \textit{uniform infinite necklaces}, which are random triangulations of a semi-infinite strip first introduced in~\cite{borot_recursive_2012}.

In Theorem \ref{thm:DecompositionTheorem}, we decompose the $\UIHPT$ into two infinite sub-maps distributed as the closed percolation hulls of the $\IIC$, and glued along a uniform necklace. The idea of such a decomposition goes back to \cite{duplantier_liouville_2014}. Together with Theorem \ref{thm:IICTheorem}, this describes how the geometry of the $\UIHPT$ is altered by the conditioning to have an infinite open percolation cluster. The $\IIC$ is obtained by cutting the $\UIHPT$ along the uniform necklace, and gluing inside, \textit{ex-nihilo}, the infinite open percolation hull.

\section{Definitions and results}\label{sec:Definitions}

\noindent\textbf{Notation.} In the following, we use the notation 
\[\N:=\{1,2,\ldots\}, \quad  \Z_+:=\N\cup\{0\}, \quad \Z_-:=\{0,-1,\ldots\} \quad \text{and} \quad \Z^*:=\Z\backslash\{0\}.\]

\subsection{Random planar maps and percolation}\label{sec:RPM}

\noindent\textbf{Maps.} A planar map is the proper embedding of a finite connected graph in the two-dimensional sphere, up to orientation-preserving homeomorphisms. For technical reasons, the planar maps we consider are always \textit{rooted}, meaning that an oriented edge called the \textit{root} is distinguished. The \textit{origin} is the tail vertex of the root. The faces of a planar map are the connected components of the complement of the embedding of the edges. The degree of a face is the number of its incident oriented edges (with the convention that the face incident to an oriented edge lies on its left). The face incident to the right of the root edge is called the \textit{root face}, and the other faces are called \textit{internal}. The set of all planar maps is denoted by $\M_f$, and a generic element of $\M_f$ is usually denoted by $\m$.

In this paper, we deal with \textit{triangulations}, which are planar maps whose faces all have degree three. We will also consider triangulations \textit{with a boundary}, in which all the faces are triangles except possibly the root face. This means that the embedding of the edges of the root face is interpreted as the boundary $\partial \m$ of $\m$. When the edges of the root face form a cycle without self-intersection, the triangulation is said to have a \textit{simple boundary}. The degree of the root face is then the \textit{perimeter} of the triangulation. Any vertex that does not belong to the root face is an \textit{inner} vertex. We make the technical assumption that triangulations are 2-connected (or type 2), meaning that multiple edges are allowed but self-loops are not.

\bigskip

\noindent\textbf{Local topology.} The \textit{local topology} on $\M_f$ is induced by the distance $\dloc$ defined by
\[\dloc(\m,\m'):=\left(1+\sup\left\lbrace R\geq 0 : \B_R(\m) \sim \B_R(\m')\right\rbrace \right)^{-1}, \quad \m,\m' \in \M_f.\] Here, $\B_R(\m)$ is the ball of radius $R$ in $\m$ for the graph distance, centered at the origin vertex. Precisely, $\B_0(\m)$ is the origin of the map, and for every $R>0$, $\B_{R}(\m)$ contains vertices at graph distance less than $R$ from the origin, and all the edges whose endpoints are in this set.

Equipped with the distance $\dloc$, $\M_f$ is a metric space whose completion is denoted by $\M$. The elements of $\M_{\infty}:=\M \setminus \M_f$ can be considered as infinite planar maps, built as the proper embedding of an infinite but locally finite graph into a non-compact surface, dissecting this surface into a collection of simply connected domains (see~\cite[Appendix]{curien_view_2013} for details). The boundary of an infinite planar map is the embedding of edges and vertices of its root face. When the root face is infinite, its vertices and edges on the left (resp.\ right) of the origin form the left (resp.\ right) boundary of the map. We use the notation $\M^\triangle$ for the set of (possibly infinite) triangulations with a boundary, and $\Ms^\triangle$ for the subset of triangulations with a simple boundary.

\bigskip

\noindent\textbf{The uniform infinite half-planar triangulation.} The study of the convergence of random planar triangulations in the local topology goes back to Angel and Schramm \cite[Theorem 1.8]{angel_uniform_2003}, whose result states as follows. For $n\in \N$, let $\P_n$ be the uniform measure on the set of rooted triangulations of the sphere with $n$ vertices. Then, in the sense of weak convergence for the local topology,
\[\P_{n} \underset{n \rightarrow \infty}{\Longrightarrow}  \P_{\infty}.\] The probability measure $\P_{\infty}$ is called (the law of) the Uniform Infinite Planar Triangulation ($\UIPT$) and is supported on infinite triangulations of the plane. An analogous result has been proved in the quadrangular case by Krikun in~\cite{krikun_local_2005}. In this setting, there exists an alternative construction using bijective techniques for which we refer to~\cite{chassaing_local_2006,menard_two_2010,curien_view_2013}.

Later on, Angel introduced in~\cite{angel_scaling_2004} a model of infinite triangulation with an infinite boundary that has nicer properties. For $n\geq 1$ and $m\geq 2$, let $\Ms^\triangle_{n,m}$ be the set of rooted triangulations of the $m$-gon (i.e.\ with a simple boundary of perimeter $m$) having $n$ inner vertices. Let $\P_{n,m}$ be the uniform probability measure on $\Ms^\triangle_{n,m}$. Then, first by \cite[Theorem 5.1]{angel_uniform_2003} and then by \cite[Theorem 2.1]{angel_scaling_2004}, in the sense of weak convergence for the local topology,
\[\P_{n,m} \underset{n \rightarrow \infty}{\Longrightarrow}  \P_{\infty,m} \quad \text{and} \quad \P_{\infty,m} \underset{m \rightarrow \infty}{\Longrightarrow} \P_{\infty,\infty}.\] The probability measure $\P_{\infty,m}$ is called (the law of) the $\UIPT$ of the $m$-gon, while $\P_{\infty,\infty}$ is (the law of) the Uniform Infinite Half-Planar Triangulation ($\UIHPT$) and is supported on infinite triangulations of the upper half-plane (as illustrated in Figure \ref{fig:UIHPTView}). A half-planar infinite triangulation should be understood as the proper embedding of an infinite but locally finite connected graph in the upper half-plane $\mathbb{H}$ such that all the faces are finite and have degree three, while the boundary is isomorphic to $\Z$. The probability measure $\P_{\infty,\infty}$ enjoys a re-rooting invariance property, in the sense that it is preserved under the natural shift operation for the root edge along the boundary. This result extends to the quadrangular case, for which an alternative construction is also provided in~\cite[Section 6.1]{curien_uniform_2015}. 

\begin{figure}[!ht]
\centering
\includegraphics[scale=1.3]{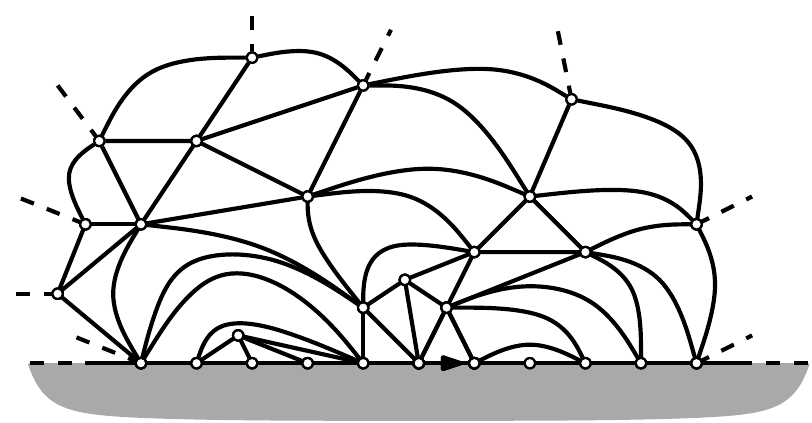}
\caption{An embedding of the $\UIHPT$ in the upper half-plane.}
\label{fig:UIHPTView}
\end{figure}

\smallskip

The properties of the $\UIHPT$ are best understood using a probability measure supported on triangulations with fixed perimeter called the \textit{Boltzmann measure} (or \textit{free measure} in~\cite{angel_growth_2003}). Let $m\geq 2$ and introduce the \textit{partition function}
\begin{equation}\label{partfct}
  Z_m:=\sum_{n\in\Z_+}{\#\Ms^\triangle_{n,m}\left(\frac{2}{27}\right)^{n}}.
\end{equation} This is the generating function of triangulations with a simple boundary of perimeter $m$ and (critical) weight $2/27$ per inner vertex. For further use, recall the asymptotics~\cite[Section 2.2]{angel_percolations_2015} \begin{equation}\label{asymptZ}
	Z_m \underset{m \rightarrow \infty}{\sim} \kappa m^{-5/2}9^m, \quad (\kappa >0).
\end{equation} The Boltzmann measure on the set $\Ms^\triangle_m$ of triangulations with a simple boundary of perimeter $m$ is defined by
\begin{equation}
\mathbb{W}_m(\m):=\frac{1}{Z_m}\left(\frac{2}{27}\right)^{n}, \quad \m\in \Ms^\triangle_{n,m}. 
\end{equation}This object is of particular importance because it satisfies a branching property, that we will identify on the $\UIHPT$ as the \textit{spatial} (or \textit{domain}) \textit{Markov property}. The tight relations between Boltzmann triangulations and the $\UIHPT$ are no coincidence, since the latter can be obtained as a limit of the first when the perimeter goes to infinity. Precisely,~\cite[Theorem 2.1]{angel_scaling_2004} states that \[\mathbb{W}_m \underset{m \rightarrow \infty}{\Longrightarrow}  \P_{\infty,\infty}\] in the sense of weak convergence, for the local topology.

\bigskip

\noindent\textbf{The spatial Markov property.} In this paragraph, we detail the so-called \textit{peeling} technique introduced by Angel. The general idea is to suppose the whole map unknown and to reveal its faces one after another. To do so, consider a map $M$ distributed as the $\UIHPT$ and the face $\A$ of $M$ incident to the root. To \textit{reveal} or \textit{peel} the face $\A$ means that we suppose the whole map unknown and work conditionally on the \textit{configuration} of this face (see the definition below). We now consider the map $M\setminus\A$, obtained by removing the root edge of $M$ (in that sense, we also say that we peel the root edge). This map has at most one cut-vertex on the boundary, which defines sub-maps that we call the (connected) components of $M\setminus\A$.

The spatial Markov property has been introduced in~\cite[Theorem 2.2]{angel_scaling_2004} and states as follows: $M\setminus\A$ has a unique infinite component $M'$ with law $\P_{\infty,\infty}$, and at most one finite component $\tilde{M}$ with law $\mathbb{W}_m$ (and perimeter $m\geq 2$ given by the configuration of the face $\A$). Moreover, $\tilde{M}$ is independent of $M'$. This is illustrated in Figure \ref{fig:SpatialMP}.

\begin{figure}[!ht]
\centering
\includegraphics[scale=1.3]{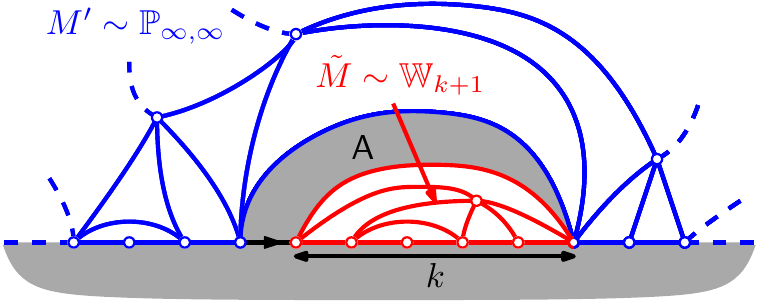}
\caption{The spatial Markov property.}
\label{fig:SpatialMP}
\end{figure}

The peeling technique is extended to a \textit{peeling process} by successively revealing a new face in the unique infinite component of the map deprived of the discovered face. The spatial Markov property ensures that the configuration of the revealed face has the same distribution at each step, while the re-rooting invariance of the $\UIHPT$ allows a complete freedom on the choice of the next edge to peel on the boundary (as long as it does not depend on the unrevealed part of the map). This is the cornerstone to study percolation on uniform infinite half-planar maps, see~\cite{angel_scaling_2004,angel_percolations_2015,richier_universal_2015}. Note that random half-planar triangulations satisfying the spatial Markov property and translation invariance have been classified in~\cite{angel_classification_2015}.

\smallskip

The peeling technique enlightens the crucial role played by the possible configurations for the face $\A$ incident to the root in the $\UIHPT$ and their probabilities. Let us introduce some notation. On the one hand, some edges of $\A$, called \textit{exposed}, belong to the boundary of the infinite component of $M\setminus\A$. On the other hand, some edges of the boundary, called \textit{swallowed}, may be enclosed in a finite component of $M\setminus\A$. The number of exposed and swallowed edges are denoted by $\Ec$ and $\Rc$. We may use the notations $\Rc_l$ and $\Rc_r$ for the number of swallowed edges on the left and on the right of the root edge. The probabilities of the two possible configurations for the face incident to the root edge in the $\UIHPT$ are provided in~\cite[Section 2.3.1]{angel_percolations_2015}:

\begin{enumerate}

\item The third vertex of $\A$ is an inner vertex $(\Ec=2,\Rc=0)$ with probability $q_{-1}=2/3$.

\item The third vertex of $\A$ is on the boundary of the map, $k\in\N$ edges on the left (or right) of the root $(\Ec=1,\Rc=k)$ with probability $q_{k}=Z_{k+1}9^{-k}$. 

\end{enumerate}

\begin{figure}[!ht]
\centering
\includegraphics[scale=1.3]{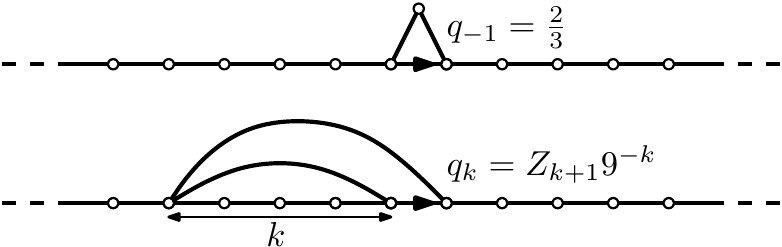}
 \caption{The configurations of the triangle incident to the root (up to symmetries).}
\label{fig:configT}
\end{figure} Note that we have $\sum_{k\in\N}{q_k}=1/6$ and $\sum_{k\in\N}{kq_k}=1/3$. By convention, we set $q_0=0$.

\bigskip

\noindent\textbf{Percolation.} We now equip the $\UIHPT$ with a Bernoulli site percolation model, meaning that every site is open (coloured black, taking value 1) with probability $p$ and closed (coloured white, taking value 0) otherwise, independently of every other site. This colouring convention is identical to that of~\cite{angel_growth_2003}, but opposed to that of~\cite{angel_percolations_2015}. Let us define the probability measure $\Pb_p$ induced by this model. For a given map $\m\in \M$, we define a measure on the set of colourings of $\m$ by
\[\mathcal{P}_p^{V(\m)}:=(p\delta_1+(1-p)\delta_0)^{\otimes {V(\m)}},\] where $V(\m)$ is the set of vertices of $\m$. Then, $\Pb_p$ is the measure on the set of coloured (or percolated) maps $\left\lbrace(\m,\mathrm{c}) : \m\in \M, \mathrm{c}\in\{0,1\}^{V(\m)}\right\rbrace$ defined by
\[\Pb_p(\mathrm{d}\m\mathrm{d}\mathrm{c}):=\P_{\infty,\infty}(\mathrm{d}\m)\mathcal{P}_p^{V(\m)}(\mathrm{d}\mathrm{c}).\] In other words, the map has the law of the $\UIHPT$ and conditionally on it, the colouring is a Bernoulli percolation with parameter $p$. We emphasize that this probability measure is annealed, so that conditioning on events depending only on the colouring may still affect the geometry of the underlying random lattice. We implicitly extend the definition of the local topology to coloured maps. In what follows, we will work conditionally on the colouring of the boundary of the map, which we call the \textit{boundary condition}. 

The (open) \textit{percolation cluster} of a vertex $v$ of the map is the set of open vertices connected to $v$ by an open path, together with the edges connecting them. If $\C$ is the open percolation cluster of the origin and $\vert \C \vert$ its number of vertices, the percolation probability is
\[\Theta(p):=\Pb_p(\vert \C \vert=\infty), \quad p\in[0,1].\] A coupling argument shows that $\Theta$ is nondecreasing, so that there exists a critical point $p_c$, called the percolation threshold, such that $\Theta(p)>0$ if $p>p_c$ and $\Theta(p)=0$ if $p<p_c$. Under the natural boundary condition that all the vertices of the boundary are closed except the origin vertex which is open, Angel proved in~\cite{angel_scaling_2004} (see also~\cite[Theorem 5]{angel_percolations_2015}) that
\[p_c=\frac{1}{2}.\] We will regularly work at criticality and use the notation $\Pb$ instead of $\Pb_{p_c}$. We slightly abuse notation here and use $\Pb$ for several boundary conditions. For every $m\geq 2$, we also denote by $\mathbf{W}_m$ the measure induced by the Bernoulli site percolation model with parameter $1/2$ on a Boltzmann triangulation with distribution $\mathbb{W}_m$ (and a boundary condition to be defined).

We end with some definition. The \textit{hull} $\Hc$ of a percolation cluster $\C$ is the coloured triangulation with a boundary obtained by filling in the finite holes of $\C$. In other words, $\Hc$ is the union of $\C$ and the finite connected components of its complement in the whole map, see Figure \ref{fig:Hull} for an example. The unique infinite connected component of the map deprived of $\C$ is called the exterior. The boundary of $\Hc$ is formed by the vertices and edges of the hull that are adjacent to the exterior or to the boundary of the map. The root edge of $\Hc$ is the rightmost edge of $\Hc$ whose origin is $v$.

\begin{figure}[!ht]
\centering
\includegraphics[scale=1.3]{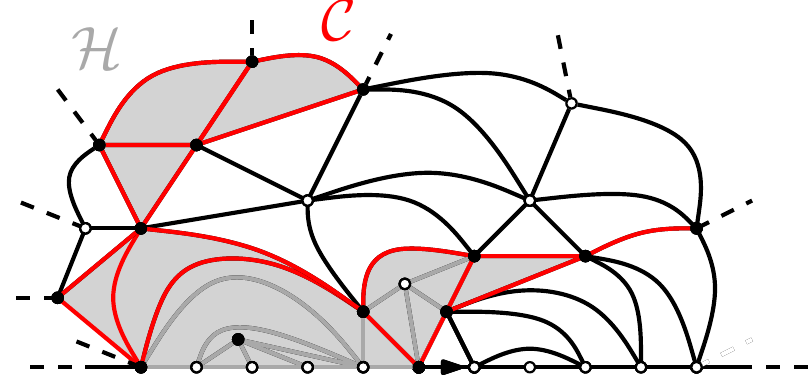}
\caption{The percolation cluster of the origin and its hull.}
\label{fig:Hull}
\end{figure}

\subsection{Random trees and looptrees}\label{sec:TreesAndLooptrees}

\noindent\textbf{Plane trees.} We use the formalism of~\cite{neveu_arbres_1986}. A finite plane tree $\tr$ as a finite subset of \[\mathcal{U}:=\bigcup_{n=0}^{\infty}{\mathbb{N}^n}\] satisfying the following properties. First, the empty word $\emptyset$ is an element of $\tr$ ($\emptyset$ is the root of $\tr$). Next, for every $n\in\N$, if $v=(v_1,\ldots,v_n)\in \tr$, then $(v_1,\ldots v_{n-1})\in \tr$ ($(v_1,\ldots v_{n-1})$ is the parent of $v$ in $\tr$). Finally, for every $v=(v_1,\ldots,v_n)\in \tr$, there exists $k_v(\tr)\in\Z_+$ such that $(v_1,\ldots,v_n,j)\in \tr$ iff $1\leq j \leq k_v(\tr)$ ($k_v(\tr)$ is the number of children of $v$ in $\tr$).

For every $v=(v_1,\ldots,v_n)\in \tr$, $\vert v \vert=n$ is the height of $v$ in $\tr$. Vertices of $\tr$ at even height are called white, and those at odd height are called black. We let $\tr_{\circ}$ and $\tr_{\bullet}$ denote the corresponding sets of vertices. The set of finite plane trees is denoted by $\T_f$. 

We will also deal with the set $\T_\textup{loc}$ of \textit{locally finite} plane trees which is the completion of $\T_f$ with respect to the local topology. Equivalently, the set $\T_\textup{loc}$ is obtained by extending the definition of $\T_f$ to infinite trees whose vertices have finite degree ($k_v(\tr)<\infty$ for every $v\in\tr$). A \textit{spine} in a tree $\tr$ is an infinite sequence $\{s_k : k\in\Z_+\}$ of vertices of $\tr$ such that $s_0=\emptyset$ and for every $k\in\Z_+$, $s_k$ is the parent of $s_{k+1}$.  

\bigskip

\noindent\textbf{Looptrees.} Looptrees have been introduced in~\cite{curien_random_2014,curien_percolation_2014} in order to study the boundary of percolation hulls in the $\UIPT$. Here, we closely follow the presentation of~\cite{curien_percolation_2014}. Let us start with a formal definition. A (finite) \textit{looptree} is a finite planar map whose edges are incident to two distinct faces, one of them being the root face (such a map is also called \textit{edge-outerplanar}). The set of finite looptrees is denoted by $\LT_f$. Informally, a looptree is a collection of simple cycles glued along a tree structure. Consistently, there is a way to construct looptrees from trees and conversely, that we now describe.

To every plane tree $\tr\in\T_f$ we associate a looptree $\lt:=\Loop(\tr)$ as follows. Vertices of $\lt$ are vertices of $\tr_\circ$, and around each vertex $u\in\tr_\bullet$, we connect the incident (white) vertices with edges in cyclic order. The looptree $\lt$ is the planar map obtained by discarding the edges of $\tr$ and its black vertices. The root edge of $\lt$ connects the origin of $\tr$ to the last child of its first offspring in $\tr$. The inverse mapping associates to a looptree $\lt\in\LT_f$ the plane tree $\tr:=\Tree(\lt)$ called the \textit{tree of components} in~\cite{curien_percolation_2014}. It is obtained by first adding an extra vertex into each inner face (or \textit{loop}) of $\lt$, and then connecting this vertex by an edge to all the vertices of the corresponding face (the edges of $\lt$ are discarded). The plane tree $\tr$ is rooted at the oriented edge between the origin of $\lt$ and the vertex lying inside the face on the left of the root edge. Our definition of looptree as well as the mappings $\Tree$ and $\Loop$ slightly differ from \cite{curien_random_2014,curien_percolation_2014}. In particular, we allow several loops to be glued at the same vertex. See Figure \ref{fig:LoopAndTree} for an illustration.

\begin{figure}[!ht]
\centering
\includegraphics[scale=1.1]{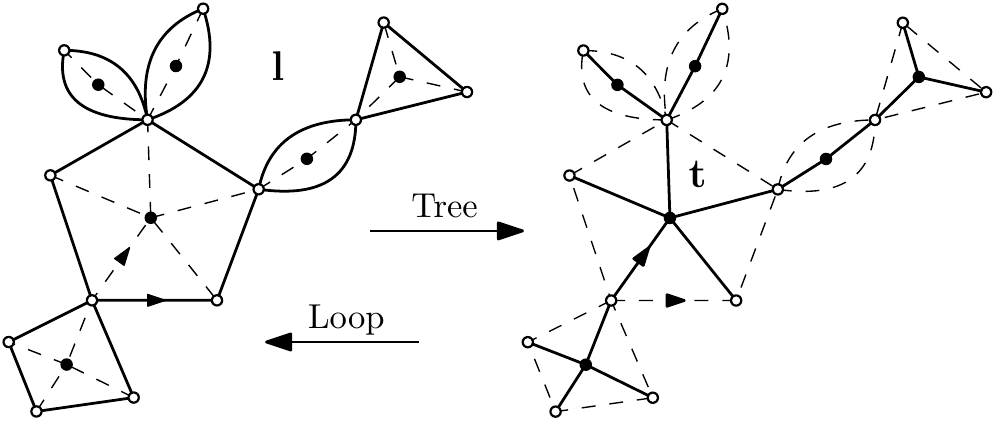}
\caption{The mappings $\Tree$ and $\Loop$.}
\label{fig:LoopAndTree}
\end{figure}

\smallskip

We now extend our definition to infinite looptrees. Formally, an infinite looptree is an edge-outerplanar map with a unique infinite face which is the root face. The set of finite and infinite looptrees is denoted by $\LT$. The application $\Loop$ extends to any locally finite plane tree $\tr\in\T_\textup{loc}$ by using the consistent sequence of planar maps $\{\Loop(\B_{2R}(\tr)) : R\in \Z_+\}$. When $\tr$ is infinite and one-ended (i.e., with a unique spine), $\Loop(\tr)$ is an infinite looptree. The inverse mapping $\Tree$ also extends to any infinite looptree $\lt\in \LT$ by using the consistent sequence of planar maps $\{\Tree(\B'_{R}(\lt)) : R \in \Z_+\}$, where $\B'_R(\lt)$ is the finite looptree made of all the internal faces of $\lt$ having a vertex at distance less than $R$ from the origin. Note that the mappings $\Tree$ and $\Loop$ are both continuous with respect to the local topology. 

\begin{Rk}\label{rk:GluingMap}
	Every internal face of a looptree $\lt$ inherits a rooting from the branching structure. Namely, the root of a loop is the edge whose origin is the closest to the origin of $\lt$, and such that the external face lies on its right. As a consequence, for every loop $l$ of perimeter $k$ in $\lt$ and every triangulation with a simple boundary $\m \in \Ms^\triangle_k$, the gluing of $\m$ in the loop $l$ is the operation which consists in identifying the boundary of $\m$ with the edges of $l$ (with the convention that the root edges are identified).
\end{Rk}

Let us now make use of these definitions to describe the branching structure of triangulations with a boundary. Following \cite[Section 2.2]{curien_uniform_2015}, we decompose any (possibly infinite) triangulation $\m$ with a (general) boundary into its \textit{irreducible components}, that is triangulations with a simple boundary attached through cut-vertices (or pinch-points) of $\partial \m$. We also define the so-called \textit{scooped-out triangulation} $\Scoop(\m)$, which is the planar map obtained by taking the boundary $\partial\m$ of $\m$ and duplicating the edges whose sides both belong to the root face of $\m$. When $\m$ has no irreducible component with an infinite boundary, $\Scoop(\m)$ is an infinite looptree and we call \textit{tree of components} of $\m$ the (locally finite) plane tree $\Treeb(\m):=\Tree(\Scoop(\m))$. See Figure \ref{fig:ScoopTreeToLoop} for an example.

\begin{figure}[!ht]
\centering
\includegraphics[scale=1.2]{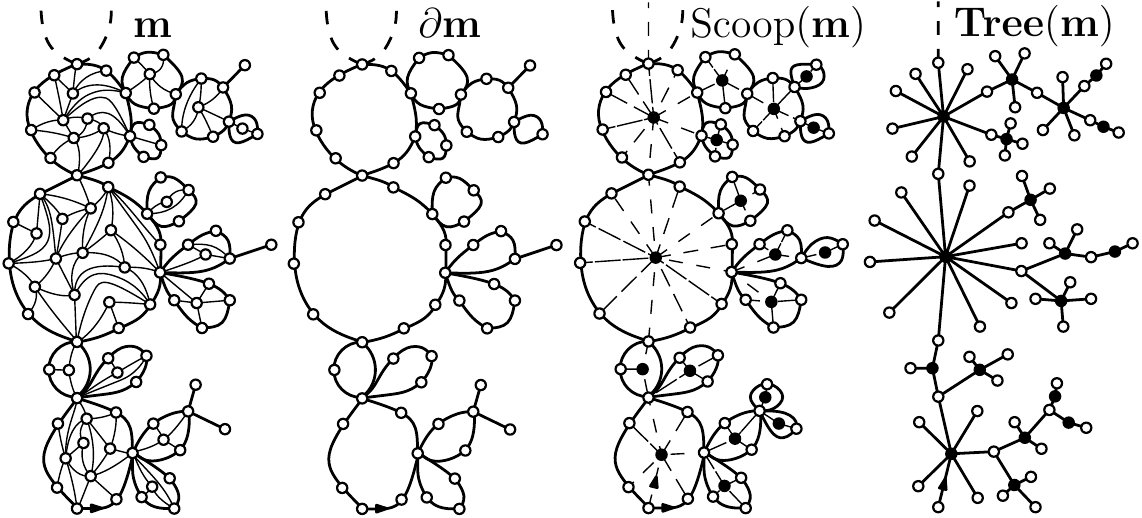}
\caption{A triangulation with a boundary $\m$, its scooped-out map $\Scoop(\m)$ and the associated tree of components $\Treeb(\m)$.}
\label{fig:ScoopTreeToLoop}
\end{figure}

\begin{Rk}The tree of components $\Treeb(\m)$ is also obtained by adding a vertex into each internal face of $\Scoop(\m)$, and connecting this vertex to all the vertices of the associated face (the edges of $\Scoop(\m)$ are erased). This definition extends to any triangulation with a boundary $\m\in\M^\triangle$. However, $\Treeb(\m)$ is not a locally finite plane tree in general, but an acyclic connected graph (with possibly vertices of infinite degree). In what follows, we only deal with triangulations $\m$ such that $\Treeb(\m)$ is locally finite and one-ended.
\end{Rk}

Let $\m\in\M^\triangle$. By construction, to every vertex $v$ at odd height in $\Treeb(\m)$ with degree $k\geq 2$ corresponds an internal face of $\Scoop(\m)$ with the same degree. Moreover, to this face is associated a triangulation $\m_v\in\Ms^\triangle_k$ with a simple boundary of perimeter $k$, which is the irreducible component of $\m$ delimited by the face. By convention, the root edge of $\m_v$ is the edge whose origin is the closest to that of $\m$, and such that the root face of $\m$ lies on its right. Then, $\m$ is recovered from $\Scoop(\m)$ by gluing the triangulation $\m_v$ in the associated face of $\Scoop(\m)$, as explained in Remark \ref{rk:GluingMap}. This results in an application
\[\Phi: \m  \mapsto \left( \tr=\Treeb(\m), \{\m_v : v\in \tr_\bullet\}  \right)\] that associates to every triangulation with a general boundary $\m\in\M^\triangle$ its tree of components $\tr=\Treeb(\m)$ together with a collection $\{\m_v : v\in \tr_\bullet\}$ of triangulations with a simple boundary having respective perimeter $\deg(v)$ attached to vertices at odd height of $\tr$, which are the irreducible components of $\m$. Note that for $\tr\in\T_\textup{loc}$, the inverse mapping $\Phi^{-1}$ that consists in filling in the loops of $\Loop(\tr)$ with the collection $\{\m_v : v\in \tr_\bullet\}$ is continuous with respect to the natural topology.

\bigskip

\noindent\textbf{Multi-type Galton-Watson trees.} Let $\nuw$ and $\nub$ be probability measures on $\Z_+$. A random plane tree is an (alternated two-type) Galton-Watson tree with offspring distribution $(\nuw,\nub)$ if all the vertices at even (resp.\ odd) height have offspring distribution $\nuw$ (resp.\ $\nub$) all independently of each other. From now on, we assume that the pair $(\nuw,\nub)$ is critical (i.e.\ its mean vector $(m_\circ,m_\bullet)$ satisfies $m_\circ m_\bullet =1$). Then, the law $\GWnn$ of such a tree is characterized by
\[\GWnn(\tr)=\prod_{v\in \tr_{\circ}}{\nuw\left(k_v(\tr)\right)}\prod_{v\in \tr_{\bullet}}{\nub\left(k_v(\tr)\right)}, \quad \tr\in \T_f.\] The construction of \textit{Kesten's tree}~\cite{kesten_subdiffusive_1986,lyons_probability_2016} has been generalized in~\cite[Theorem 3.1]{stephenson_local_2016} to multi-type Galton-Watson trees conditioned to survive as follows. Assume that the critical pair $(\nuw,\nub)$ satisfies $\GWnn(\{\vert \tr \vert =n\})>0$ for every $n\in\Z_+$. Let $T_n$ be a plane tree with distribution $\GWnn$ conditioned to have $n$ vertices. Then, in the sense of weak convergence, for the local topology 
\[T_n \underset{n \rightarrow \infty}{\overset{(d)}{\longrightarrow}}  \Tinf.\] The random infinite plane tree $\Tinf=\Tinf(\nuw,\nub)$ is a multi-type version of Kesten's tree, whose law is denoted by $\GWnn^{(\infty)}$. Let us describe the alternative construction of $\Tinf$ as explained in~\cite{stephenson_local_2016}. For every probability measure $\nu$ on $\Z_+$ with mean $m\in(0,\infty)$, the size-biased distribution $\bar{\nu}$ reads 
\[\bar{\nu}(k):=\frac{k\nu(k)}{m}, \quad k\in\Z_+.\] The tree $\Tinf$ has a.s. a unique spine, in which white vertices have offspring distribution $\bar{\nu}_\circ$ while black vertices have offspring distribution $\bar{\nu}_{\bullet}$. Each vertex of the spine has a unique child in the spine, chosen uniformly at random among the offspring. Out of the spine, white and black vertices have offspring distribution $\nuw$ and $\nub$ respectively, and the number of offspring are all independent. 

We will use two variants of $\Tinf$, which are obtained by discarding all the vertices and edges on the left (resp.\ right) of the spine, excluding the children of black vertices of the spine. Their distributions are denoted by $\GWnn^{(\infty,l)}$ and $\GWnn^{(\infty,r)}$ respectively. The infinite looptree $\Linf(\nuw,\nub):=\Loop(\Tinf)$ plays a special role in the following, and may be called \textit{Kesten's looptree} with offspring distribution $(\nuw,\nub)$. It has a unique spine of finite loops (a collection of (incident) internal faces that all disconnect the root from infinity), associated to black vertices of the spine of $\Tinf$.

\subsection{Statement of the results}\label{sec:StatementResults}

\noindent\textbf{Uniform infinite necklace.} A necklace is a map that was first introduced in~\cite{borot_recursive_2012} for studying the $O(n)$ model on random maps, see also~\cite{curien_percolation_2014}. Formally, an infinite necklace is a locally finite triangulation of the upper half-plane $\mathbb{H}$ with no inner vertex.

Consider the graph of $\Z$ embedded in the plane, and rooted at the oriented edge $(0,1)$. Let $(z_i : i\in\N)$ be a sequence of independent random variables with Bernoulli distribution of parameter $1/2$, and define the simple random walk
\[S_k:=\sum_{i=1}^{k}{z_i}, \quad k\in\N.\] The uniform infinite necklace is the random rooted map obtained from $\Z$ by adding the set of edges $\left\lbrace \left( -S_k,k+1-S_k \right) : k\in\N \right\rbrace$ in a non-crossing manner.  It is a.s.\ an infinite necklace in the aforementioned sense, and can also be interpreted as a gluing of triangles along their sides, with the tip oriented to the left or to the right equiprobably and independently. Its distribution is denoted by $\mathsf{UN}(\infty,\infty)$. See Figure \ref{fig:Necklace} for an illustration.

\begin{figure}[ht]
\centering
\includegraphics[scale=1.3]{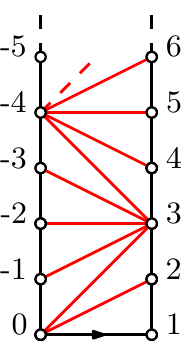}
\caption{The uniform infinite necklace.}
\label{fig:Necklace}
\end{figure}

\smallskip

In the next part, we will perform gluing operations of triangulations with a boundary along infinite necklaces. Let $\m$ and $\m'$ be triangulations with an infinite boundary. Let $\{e_i : i\in\N\}$ be the sequence of half-edges of the root face of $\m$ on the right of the origin vertex, listed in contour order. Similarly, the left boundary of $\m'$ defines the sequence of half-edges $\{e'_i : i\in\N\}$. Let $\mathbf{n}$ be an infinite necklace, with a boundary identified to $\Z$. The gluing $\Psi_\mathbf{n}(\m,\m')$ of $\m$ and $\m'$ along $\mathbf{n}$ is the map defined as follows. For every $i\in\N$, we identify the half-edge $(i,i+1)$ of $\mathbf{n}$ with $e'_i$, and the half-edge $(-i,-i+1)$ of $\mathbf{n}$ with $e_i$. The root edge of $\Psi_\mathbf{n}(\m,\m')$ is the root edge of $\mathbf{n}$. An example is given in Figure \ref{fig:DecompositionUIHPT}. 

Note that $\Psi_\mathbf{n}(\m,\m')$ is still a triangulation with an infinite boundary. In particular, our construction extends to the gluing $\Psi_{(\mathbf{n},\mathbf{n}')}(\m,\m',\m'')$ of three rooted triangulations with an infinite boundary $\m$, $\m'$ and $\m''$ along the pair of infinite necklaces $(\mathbf{n},\mathbf{n}')$. To do so, first define the triangulation with an infinite boundary $\m^*:=\Psi_\mathbf{n}(\m,\m')$, but keep the root edge of $\m'$ as the root edge of $\m^*$. Then, set $\Psi_{(\mathbf{n},\mathbf{n}')}(\m,\m',\m''):=\Psi_{\mathbf{n}'}(\m^*,\m'')$. See Figure \ref{fig:DecompositionIIC} for an example. These gluing operations are continuous with respect to the local topology.

\bigskip

\noindent\textbf{Decomposition of the $\normalfont{\UIHPT}$.} We consider the $\UIHPT$ decorated with a critical percolation model, and work conditionally on the ``Black-White" boundary condition of Figure \ref{fig:InitialColouring}. We let $\Hb$ and $\Hw$ be the hulls of the percolation clusters of the origin and the target of the root. We denote by $\Treeb(\Hb)$ and $\Treeb(\Hw)$ their respective tree of components, and by 
\[\left\lbrace M_v^b : v\in \Treeb(\Hb)_\bullet \right\rbrace \quad \text{and} \quad \left\lbrace M_v^w : v\in \Treeb(\Hw)_\bullet \right\rbrace\] their irreducible components (i.e.\ the second components of $\Phi(\Hb)$ and $\Phi(\Hw)$). The boundary conditions of the irreducible components are determined by the hull. We define the probability measures $\muw$ and $\mub$ by
\begin{equation}\label{eqn:DefinitionMuwMub}
	\muw(k):=\frac{2}{3}\left(\frac{1}{3}\right)^k \quad \text{and} \quad \mub(k):=6q_k, \quad k\in\Z_+.  
\end{equation}

\begin{figure}[!ht]
\centering
\includegraphics[scale=1.6]{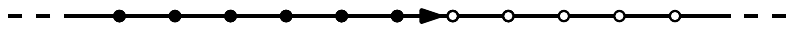}
\caption{The ``Black-White" boundary condition.}
\label{fig:InitialColouring}
\end{figure}

\begin{Th}\label{thm:DecompositionTheorem}In the critical Bernoulli percolation model on the $\UIHPT$ with ``Black-White" boundary condition:

\begin{itemize}
	\item The trees of components $\Treeb(\Hb)$ and $\Treeb(\Hw)$ are independent with respective distribution $\mathsf{GW}^{(\infty,l)}_{\muw,\mub}$ and $\mathsf{GW}^{(\infty,r)}_{\muw,\mub}$.
	\item Conditionally on $\Treeb(\Hb)$ and $\Treeb(\Hw)$, the irreducible components $\{M^b_v : v \in \Treeb(\Hb)_\bullet\}$ and $\{M^w_v : v \in \Treeb(\Hw)_\bullet\}$ are independent critically percolated Boltzmann triangulations with a simple boundary and respective distribution $\mathbf{W}_{\deg(v)}$.
	
\end{itemize} Finally, the $\UIHPT$ is recovered as the gluing $\Psi_\mathbf{N}(\Hb,\Hw)$ of $\Hb$ and $\Hw$ along a uniform infinite necklace $\mathbf{N}$ with distribution $\mathsf{UN}(\infty,\infty)$ independent of $(\Hb,\Hw)$.
\end{Th} 

\begin{Rk} The result of Theorem \ref{thm:DecompositionTheorem} can be seen as a discrete counterpart to~\cite[Theorem 1.16-1.17]{duplantier_liouville_2014}, as we will discuss in Section \ref{sec:ScalingLimits}. It could also be stated without reference to percolation: By discarding the colouring of the vertices, we obtain a decomposition of the $\UIHPT$ into two independent looptrees filled in with Boltzmann triangulations and glued along a uniform necklace. An illustration is provided in Figure \ref{fig:DecompositionUIHPT}.
 \end{Rk}

\begin{figure}[!ht]
\centering
\includegraphics[scale=1.2]{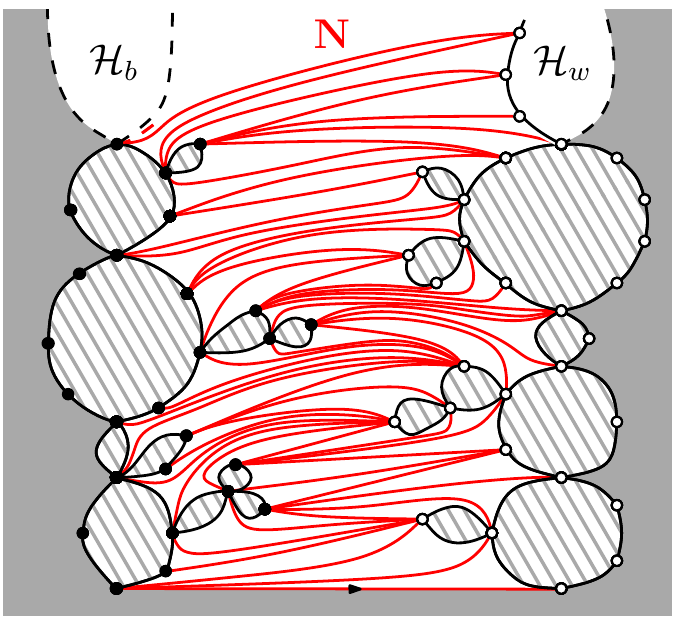}
\caption{The decomposition of the $\UIHPT$ into its percolation hulls and the uniform infinite necklace. The hatched areas are filled-in with independent critically percolated Boltzmann triangulations of the given boundary length.}
\label{fig:DecompositionUIHPT}
\end{figure}

\noindent\textbf{The incipient infinite cluster.} We now consider the $\UIHPT$ decorated with a Bernoulli percolation model with parameter $p$ conditionally on the ``White-Black-White" boundary condition of Figure \ref{fig:InitialColouringIIC}. In a map with such a boundary condition, we let $\Hc$, $\Hl$ and $\Hr$ be the hulls of the clusters of the origin and its left and right neighbours on the boundary. We let $\Treeb(\Hc)$, $\Treeb(\Hl)$ and $\Treeb(\Hr)$ be their respective tree of components, and denote by 
\[\left\lbrace M_v : v\in \Treeb(\Hc)_\bullet \right\rbrace, \quad \left\lbrace M_v^l : v\in \Treeb(\Hl)_\bullet \right\rbrace \quad \text{and} \quad \left\lbrace M_v^r : v\in \Treeb(\Hr)_\bullet \right\rbrace\] their irreducible components (i.e.\ the second components of $\Phi(\Hc)$, $\Phi(\Hl)$ and $\Phi(\Hr)$). Again, the boundary conditions of these components are determined by the hulls.

The \textit{height} $h(\C)$ of the open percolation cluster of the origin $\C$ will be defined in Section \ref{sec:ExplorationProcessIIC}. It corresponds to the maximal length of the open segment revealed when exploring the percolation interface between the origin and its left neighbour on the boundary.

\begin{figure}[!ht]
\centering
\includegraphics[scale=1.6]{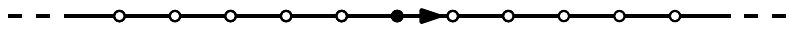}
\caption{The ``White-Black-White" boundary condition.}
\label{fig:InitialColouringIIC}
\end{figure}

\begin{Th}\label{thm:IICTheorem} Let $\Pb_p$ be the law of the $\UIHPT$ with ``White-Black-White" boundary condition equipped with a Bernoulli percolation model with parameter $p$. Then, in the sense of weak convergence for the local topology
\[\Pb_p(\cdot \mid \vert \C \vert =\infty) \underset{p \downarrow p_c}{\Longrightarrow}  \Piic\quad \text{and} \quad \Pb_{p_c}(\cdot \mid h(\C) \geq n) \underset{n \rightarrow \infty}{\Longrightarrow}  \Piic.\] The probability measure $\Piic$ is called (the law of) the Incipient Infinite Cluster of the $\UIHPT$ or $\IIC$.
The $\IIC$ is a.s. a percolated triangulation of the half-plane with ``White-Black-White" boundary condition. Moreover, in the $\IIC$:

\begin{itemize}
	\item The trees of components $\Treeb(\Hc)$, $\Treeb(\Hl)$ and $\Treeb(\Hr)$ are independent with respective distribution $\mathsf{GW}^{(\infty)}_{\muw,\mub}$, $\mathsf{GW}^{(\infty,l)}_{\muw,\mub}$ and $\mathsf{GW}^{(\infty,r)}_{\muw,\mub}$.
	
	\item Conditionally on $\Treeb(\Hc)$, $\Treeb(\Hl)$ and $\Treeb(\Hr)$, the irreducible components $\{M_v : v \in \Treeb(\Hc)_\bullet\}$, $\{M_v^l : v \in \Treeb(\Hl)_\bullet\}$ and $\{M^r_v : v \in \Treeb(\Hr)_\bullet\}$ are independent critically percolated Boltzmann triangulations with a simple boundary and respective distribution $\mathbf{W}_{\deg(v)}$.

\end{itemize} 

Finally, the $\IIC$ is recovered as the gluing $\Psi_{(\mathbf{N}_l,\mathbf{N}_r)}(\Hl,\Hc,\Hr)$ of $\Hl$, $\Hc$ and $\Hr$ along an pair of independent uniform infinite necklaces $(\mathbf{N}_l,\mathbf{N}_r)$ with distribution $\mathsf{UN}(\infty,\infty)$, also independent of $(\Hl,\Hc,\Hr)$. (The root edge of the $\IIC$ connects the origin of $\Hc$ to that of $\Hr$.)\end{Th} 

\begin{Rk}In the work of Kesten~\cite{kesten_incipient_1986}, the $\IIC$ is defined for bond percolation on $\Z^2$ by considering a supercritical open percolation cluster, and letting $p$ decrease towards the critical point $p_c$. Equivalently, the $\IIC$ arises directly in the critical model when conditioning the open cluster of the origin to reach the boundary of $[-n,n]^2$, and letting $n$ go to infinity. Theorem \ref{thm:IICTheorem} is the analogous result for site percolation on the $\UIHPT$; however, we use a slightly different conditioning in the critical setting, which is more adapted to the use of the peeling techniques.
\end{Rk}

\begin{Rk}Theorem \ref{thm:IICTheorem} should be seen as a counterpart to Theorem \ref{thm:DecompositionTheorem}. Indeed, the decomposition of the $\IIC$ shows that when conditioning the open cluster of the origin to be infinite, one adds ex-nihilo an infinite looptree in the $\UIHPT$, as shown in Figure \ref{fig:DecompositionIIC}. This describes how the zero measure event we condition on twists the geometry of the initial random half-planar triangulation.\end{Rk}

\begin{figure}[!ht]
\centering
\includegraphics[scale=1.2]{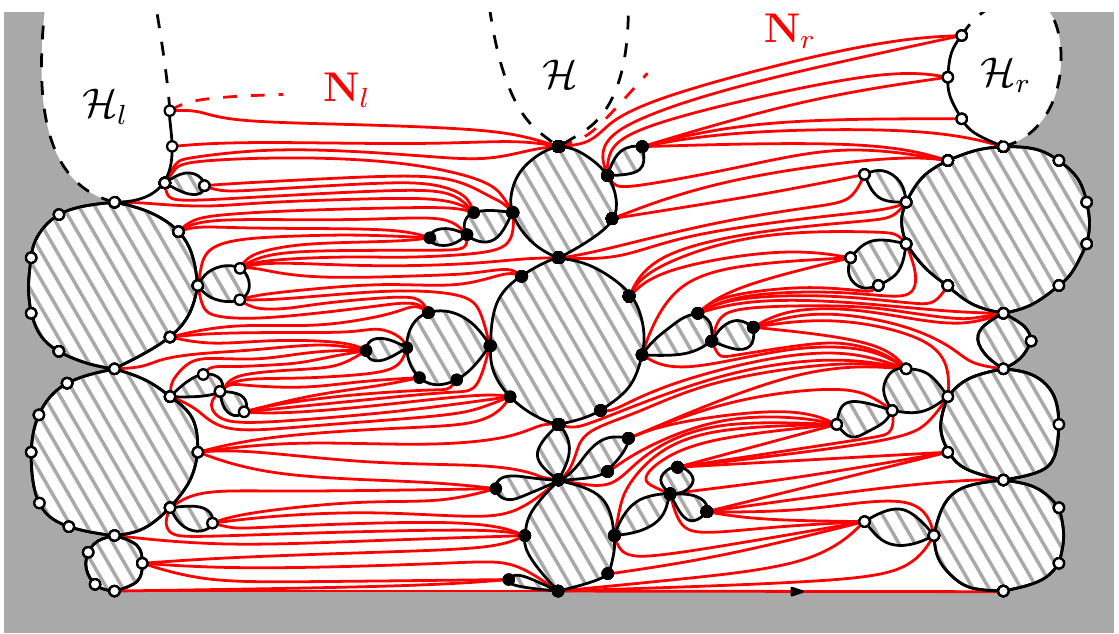}
\caption{The decomposition of the $\IIC$ into its percolation hulls and the uniform necklaces. The hatched areas are filled-in with independent critically percolated Boltzmann triangulations of the given boundary length.}
\label{fig:DecompositionIIC}
\end{figure}

\section{Coding of looptrees}\label{sec:CodingLooptrees}

\subsection{The contour function}\label{sec:ContourFunction}

We first describe the encoding of looptrees via an analogue of the contour function for trees~\cite{aldous_continuum_1993,le_gall_random_2005}. This bears similarities with the coding of continuum random looptrees of~\cite{curien_percolation_2014}.

\bigskip

\noindent\textbf{Finite looptrees.} Let $n\in\Z_+$ and $C=\{C_k : 0\leq k \leq n\}$ a discrete \textit{excursion} with no positive jumps and no constant step, that is $C_0=C_n=0$, and for every $0\leq k < n$, $C_k\in\Z_+$, $C_{k+1}-C_k\leq 1$ and $C_k\neq C_{k+1}$. 
The equivalence relation $\sim$ on $\{0,\ldots,n\}$ is defined by
\begin{equation}\label{eqn:EquivC}
	i \sim j \quad \text{iff} \quad C_i=C_j=\inf_{i \wedge j \leq k \leq i \vee j}{C_k}.
\end{equation} The quotient space $\{0,\ldots,n\}/\sim$ inherits the graph structure of the chain $\{0,\ldots,n\}$, and can be embedded in the plane as follows. Consider the graph of $C$ (with linear interpolation), together with the set of edges $E$ containing all the pairs $\{(i,C_i),(j,C_j)\}$ such that $ i \sim j$. This defines a planar map $\m_C$, whose vertices are identified to $\{0,\ldots,n\}$. The root edge of $\m_C$ connects the vertex $n$ to $n-1$. The embedding of $\{0,\ldots,n\}/\sim$ is obtained by contracting the edges of $E$ as in Figure \ref{fig:FiniteLooptree}. We obtain a looptree denoted by $\Lt_C$. Let us describe the tree of components $\Tr_C:=\Tree(\Lt_C)$. The black vertices of $\Tr_C$ are the internal faces of $\m_C$, and the white vertices of $\Tr_C$ are the equivalence classes of $\sim$. For every black vertex $f$, the white vertices incident to $f$ are the classes that have a representative $j$ incident to the face $f$ in $\m_C$, with the natural cyclic order. The root edge of $\Tr_C$ connects the class of $0$ to the the face on the left of the root of $\m_C$. See Figure \ref{fig:FiniteLooptree} for an example. This construction extends to the case where $C_n>0$, but the resulting map in not a looptree in general.

\begin{figure}[!ht]
\centering
\includegraphics[scale=1.5]{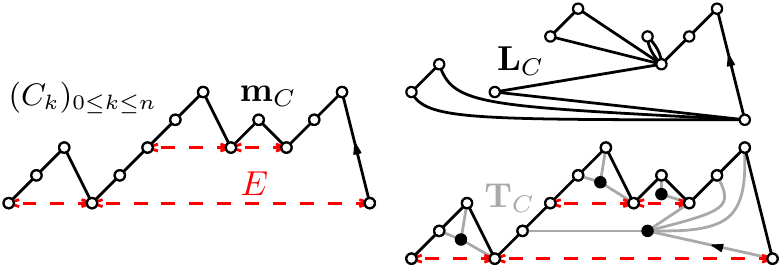}
\caption{The construction of $\Lt_C$ and $\Tr_C$ from the excursion $C$.}
\label{fig:FiniteLooptree}
\end{figure}

\bigskip

\noindent\textbf{Infinite looptrees.} Let us extend the construction to infinite looptrees. Let $C:\Z_+ \rightarrow \Z$ be a function such that $C_0=0$, with no positive jump, no constant step and such that
\begin{equation}\label{eqn:LimInf}
	\liminf_{k \rightarrow \infty}C_k = -\infty.
\end{equation} The function $C$ is extended to $\Z$ by setting $C_k=k$ for $k\in\Z_-$. We define an equivalence relation $\sim$ on $\Z$ by applying (\ref{eqn:EquivC}) with the function $C$. The graph of $-C$ and the set of edges $E$ containing all the pairs $\{(i,-C_i),(j,-C_j)\}$ such that $ i \sim j$ define an infinite map $\m_C$, whose vertices are identified to $\Z$ (the root connects $0$ to $1$). By contracting the edges of $E$, we obtain the infinite looptree $\Lt_C$ (which is an embedding of $\Z/\sim$). The tree $\Tr_C:=\Tree(\Lt_C)$ is defined as in the finite setting. By the assumption (\ref{eqn:LimInf}), internal faces of $\m_C$ (the black vertices) and equivalence classes of $\Z/\sim$ (the white vertices) are finite. Thus, $\Tr_C$ is locally finite. We let
\begin{equation}\label{eqn:ExcursionIntervals}
	\tau_0=0 \quad \text{and} \quad \tau_{k+1}:=\inf \left\lbrace i\geq \tau_k : C_i<C_{\tau_k} \right\rbrace, \quad k\in\Z_+.
\end{equation} For every $k\in\N$, the white vertex of $\Tr_C$ associated to $\tau_k$ disconnects the root from infinity, as well as its (black) parent in $\Tr_C$. This exhibits the unique spine of $\Tr_C$ (and a spine of faces in $\Lt_C$). Since $C_k=k$ for negative $k$, there is no vertex on the left of the spine of $\Lt_C$.

\begin{figure}[ht]
\centering
\includegraphics[scale=1.5]{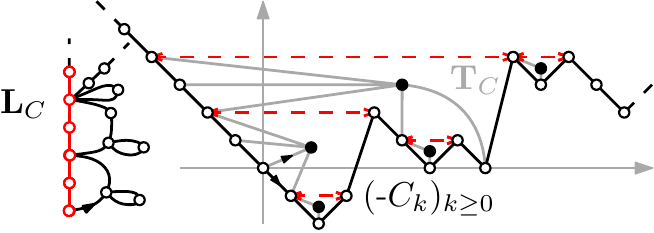}
\caption{The construction of $\Lt_C$ and $\Tr_C$ from $C$.}
\label{fig:HalfInfiniteLooptree}
\end{figure}

\medskip

We finally define a looptree out of a pair of functions $C,C':\Z_+ \rightarrow \Z$, with $C_0=C'_0=0$, no positive jumps, no constant steps and so that $C$ satisfies (\ref{eqn:LimInf}) and $C'$ is nonnegative. First define a looptree $\Lt_C$ as above, and let $\{e_i : i\in \Z_+\}$ be the half-edges of the left boundary of $\Lt_C$ in contour order. Then, we define an equivalence relation $\sim$ on $\Z_+$ by applying (\ref{eqn:EquivC}) with the function $C'$. Let $R_0=-1$ and for every $k\in\N$, $ R_k:=\sup\{i\in\Z_+ : C'_i = k-1\}$. For every $k\in\Z_+$, the excursion $\{C'_{R_{k}+i+1}-k : 0\leq i \leq R_{k+1}-R_k-1\}$ of $C'$ above its future infimum defines a looptree $\Lt_k$. We now consider the graph of $\Z_+$ embedded in the plane and for every $k\in\Z_+$, attach the looptree $\Lt_k$ on the left of the vertex $k\in\Z_+$ (so that the origin of $\Lt_k$ matches the vertex $k$ and its root edge follows $(k,k+1)$ in counterclockwise order). We obtain a \textit{forest} of looptrees $\mathbf{F}_{C'}$, isomorphic to $\Z_+/\sim$. The infinite looptree $\Lt_{C,C'}$ is obtained by gluing the left boundary of $\Lt_C$ to the right boundary of $\mathbf{F}_{C'}$ (i.e., by identifying the half-edge $e_i$ with the half-edge $(i+1,i)$ of $\N$ for every $i\in\N$). The root edge of $\Lt_{C,C'}$ is the root edge of $\Lt_{C}$. The tree of components $\Tr_{C,C'}:=\Tree(\Lt_{C,C'})$ has a unique spine inherited from $\Tr_C$. We can also define a function $C^*:\Z\rightarrow\Z$ by
\begin{equation}\label{eqn:ProcessCStar}
C^*_k=\left\lbrace
\begin{array}{ccc}
-C_{k} & \mbox{if} & k\in \Z_+\\
C'_{-k} & \mbox{if} & k\in \Z_-\\
\end{array}\right..
\end{equation} and an equivalence relation $\sim$ on $\Z$ by
\begin{equation}\label{eqn:EquivCStar}
  i \sim j \quad \text{iff} \quad
\left\lbrace\begin{array}{ccc}
i \vee j > 0 & \mbox{and} & \inf_{k \leq i \wedge j }{C^*_k} \geq C^*_i=C^*_j \geq \sup_{(i \wedge j)\vee 0 \leq k \leq i \vee j}{C^*_k}\\
\mbox{or} & & \\
i \vee j \leq 0 & \mbox{and} & \inf_{i \wedge j \leq k \leq i \vee j}{C^*_k} \geq C^*_i=C^*_j \\
\end{array}\right.
\end{equation} (with $\inf {\emptyset} = +\infty $ and $\sup {\emptyset} = -\infty $). Then, $\Lt_{C,C'}$ is isomorphic to $\Z/\sim$ (see Figure \ref{fig:InfiniteLooptree}). In the next part, we let $p_C$ and $p_{C,C'}$ denote the canonical projection on $\Lt_{C}$ and $\Lt_{C,C'}$.

\begin{figure}[!ht]
\centering
\includegraphics[scale=1.5]{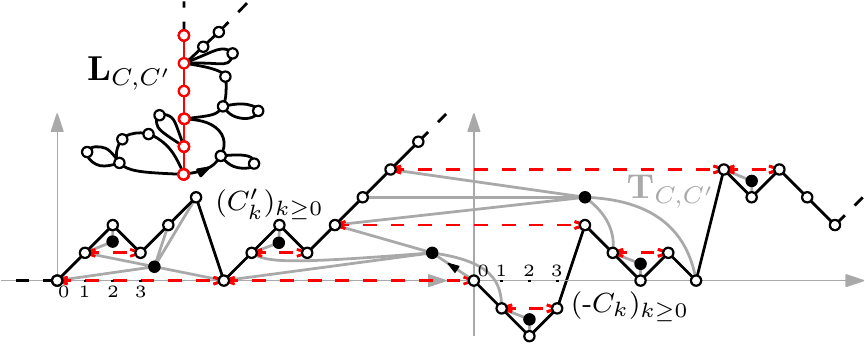}
\caption{The construction of $\Lt_{C,C'}$ and $\Tr_{C,C'}$ from $C$ and $C'$.}
\label{fig:InfiniteLooptree}
\end{figure}

\subsection{Random walks}\label{sec:RandomWalk}

We now gather results on random walks. For every probability measure $\nu$ on $\Z$ and every $x\in \Z$, let $P^{\nu}_x$ be the law of the simple random walk started at $x$ with step distribution $\nu$ (we may omit the exponent $\nu$). Let $Z=\{Z_k : k\in\Z_+\}$ be the canonical process and $T:=\inf \left\lbrace k\in\Z_+ : Z_k<0\right\rbrace$. We assume that $\nu$ is centered and $\nu((1,+\infty))=0$ (the random walk is called upwards-skip-free or with no positive jumps). For every $k\in \Z$, we let $\widehat{\nu}(k)=\nu(-k)$. 

\bigskip

\noindent\textbf{Overshoot.} We start with a result on the \textit{overshoot} at the first entrance in $(-\infty,0)$.

\begin{Lem}\label{lem:RandomWalkLastJumpTimeReverse} We have
	\[P^{\nu}_0\left( Z_{T-1}-Z_T=k \right)=\frac{k\nu(-k)}{\nu(1)}, \quad k\in\Z_+.\] Conditionally on $Z_{T-1}-Z_T$, $-Z_{T}$ is uniform on $\left\lbrace 1, \ldots, Z_{T-1}-Z_T \right\rbrace$. Moreover, under $P^{\nu}_0$ and conditionally on $Z_{T-1}$, the reversed process $\{\widehat{Z}_k :0\leq k < T\}:=\{Z_{T-1-k} :0\leq k < T\}$ has the same law as $\{Z_{k} : 0\leq k <T\}$ under $P^{\widehat{\nu}}_{Z_{T-1}}$. 
\end{Lem}
 
\begin{proof} Let $n\in\N$ and $x_0,\ldots,x_{n-1} \geq 0$. On the one hand
\begin{align*}
	P^{\nu}_0(T=n,\widehat{Z}_0=x_0,\ldots,\widehat{Z}_{n-1}=x_{n-1})&=P^{\nu}_0\left(Z_0=x_{n-1},\ldots,Z_{n-1}=x_0,Z_n<0\right)\\
	&=\mathbf{1}_{\left\lbrace x_{n-1}=0 \right \rbrace}\nu(x_{n-2}-x_{n-1})\cdots \nu(x_0-x_1)\nu((-\infty,-x_0))
\end{align*} and on the other hand since $\nu((1,+\infty))=0$,
\begin{align*}
	P^{\widehat{\nu}}_{x_0}&\left(T=n, Z_0=x_0,\ldots,Z_{n-1}=x_{n-1}\right)=\mathbf{1}_{\left\lbrace x_{n-1}=0 \right \rbrace}\nu(-(x_1-x_0))\cdots \nu(-(x_{n-1}-x_{n-2}))\nu(1).
\end{align*}Now, while computing the probability $P^{\nu}_0\left(Z_{T-1}=x_{0} \right)$ one gets 
\begin{align*}P^{\nu}_0\left(Z_{T-1}=x_{0} \right)=\nu((-\infty,-x_0))\sum_{n\in\N}{\sum_{x_1,\ldots,x_{n-1}\geq 0}{\mathbf{1}_{\left\lbrace x_{n-1}=0 \right \rbrace}\nu(x_{n-2}-x_{n-1})\cdots \nu(x_0-x_1)}},\end{align*}and still using $\nu((1,+\infty))=0$, we have

\begin{align*}1=P^{\widehat{\nu}}_{x_0}\left(Z_{T-1}=0 \right)=\nu(1)\sum_{n\in\N}{\sum_{x_1,\ldots,x_{n-1}\geq 0}{\mathbf{1}_{\left\lbrace x_{n-1}=0 \right \rbrace}\nu(x_{n-2}-x_{n-1})\cdots \nu(x_0-x_1)}}.\end{align*} The last assertion follows, as well as $P^{\nu}_0\left(Z_{T-1}=i \right)=\nu((-\infty,-i))/\nu(1)$ for every $i\in\Z_+$. By a direct computation, for every $i\in\Z_+$ and $j\in \Z_-\backslash\{0\}$,
\[P^{\nu}_0\left(Z_{T}=j \mid Z_{T-1}=i \right)=\frac{\nu(j-i)}{\nu((-\infty,-i))} \quad \text{and} \quad  P^{\nu}_0\left(Z_{T-1}=i,Z_{T}=j\right)=\frac{\nu(j-i)}{\nu(1)}.\] Then, for every $k\in\Z_+$ and $l\in\{1,\ldots,k\}$,
\[P^{\nu}_0\left(Z_{T-1}-Z_{T}=k\right)=\frac{k\nu(-k)}{\nu(1)} \quad \text{and} \quad P^{\nu}_0\left(-Z_{T}=l\mid Z_{T-1}-Z_{T}=k\right)=\frac{1}{k},\] which ends the proof.\end{proof}

\begin{Rk}Since $\nu((1,+\infty))=0$, by putting $\widehat{Z}_T:=\widehat{Z}_{T-1}-1$ we have that under $P^{\nu}_0$ and conditionally on $Z_{T-1}$, $\{\widehat{Z}_k : 0\leq k \leq T\}$ is distributed as $\{Z_k : 0\leq k \leq T\}$ under $P^{\widehat{\nu}}_{Z_{T-1}}$.
\end{Rk}

\bigskip

\noindent\textbf{Random walk conditioned to stay nonnegative.} We now recall the construction of the so-called \textit{random walk conditioned to stay nonnegative} of \cite{bertoin_conditioning_1994} (see also~\cite{feller_introduction_1971,spitzer_principles_2001}). We let $T_n:=\inf \left\lbrace k\in\N : Z_k\geq n \right\rbrace$ for every $n\in\Z_+$. Let $H:=\{H_k : k\in\Z_+\}$ be the strict ascending ladder height process of $-Z$. Namely, let $L_0=0$ and 
\[H_k=-Z_{L_k}, \quad L_{k+1}=\inf\{j>L_k : -Z_j>H_k\}, \quad k\in\Z_+.\] Then, the \textit{renewal function} associated with $H$ is defined by 
\begin{align}\label{eqn:RenewalFunction}
	V(x):=\sum_{k=0}^\infty{P_0(H_k\leq x)}=E_0\left(\sum_{k=0}^{T_0-1}{\mathbf{1}_{\{Z_k\geq -x\}}}\right), \quad x\geq 0,
\end{align} where the equality follows from the duality lemma. For every $x\geq 0$, we denote by $P_x^{\uparrow}$ the (Doob) h-transform of $P_x$ by $V$. That is, for every $k\in\Z_+$ and every $F:\Z^{k+1}\rightarrow \R$, 
\begin{align}\label{eqn:HTransform}
	E_x^\uparrow(F(Z_0,\ldots,Z_k))=\frac{1}{V(x)}E_x\left(V(Z_k)F(Z_0,\ldots,Z_k)\mathbf{1}_{k<T}\right).
\end{align} 

\begin{Thbis}{\textup{\cite[Theorem 1]{bertoin_conditioning_1994}}} For every $x\geq 0$, in the sense of weak convergence of finite-dimensional distributions, \[P_x\left( \cdot \mid T_n<T \right) \underset{n \rightarrow +\infty}{\Longrightarrow} P_x^{\uparrow} \quad \text{ and } \quad P_x\left( \cdot \mid T\geq n \right) \underset{n \rightarrow +\infty}{\Longrightarrow} P_x^{\uparrow}.\] The probability measure $P_x^{\uparrow}$ is the law of the random walk conditioned to stay nonnegative.	
\end{Thbis}

We now recall Tanaka's pathwise construction, and let $T_+:=\inf \left \lbrace k\in\Z_+ : Z_k>0 \right \rbrace$.

\begin{Thbis}{\textup{\cite[Theorem 1]{tanaka_time_1989}}} Let $\{w_k : k\in\Z_+\}$ be independent copies of the reversed excursion
\[\left(0,Z_{T_+}-Z_{T_+-1},\ldots,Z_{T_+}-Z_{1},Z_{T_+}\right)\] under $P_0$, with $w_k=(w_k(0),\ldots,w_k(s_k))$. Let for every $k\in\Z_+$
\[Y_k':= \sum_{j=0}^{i-1}{w_j(s_j)}+w_i\left(k-\sum_{j=0}^{i-1}{w_j(s_j)}\right) \quad \text{for} \quad \sum_{j=0}^{i-1}{w_j(s_j)}<k\leq \sum_{j=0}^{i}{w_j(s_j)},\] and $Y_k:=Y'_{k+1}-1$. Then, the process $\{Y_k : k\in\Z_+\}$ has law $P_0^{\uparrow}$.
	
\end{Thbis}

\begin{Rk}\label{rk:RWCSN}In \cite{tanaka_time_1989}, $\{Y_k' : k\in\Z_+\}$ is the h-transform of $P_0$ by a suitable renewal function $V'$. This function differs from the function $V$ of \eqref{eqn:RenewalFunction} and rather defines a random walk conditioned to stay positive. However, when the random walk is upwards-skip-free and we remove its first step (which gives $\{Y_k : k\in\Z_+\}$), the associated renewal function equals $V$ up to a multiplicative constant. This ensures that $\{Y_k : k\in\Z_+\}$ has law $P_0^{\uparrow}$.\end{Rk}

\noindent Let us rephrase this theorem. Let $R_0=-1$ and $R_k:=\sup \left \lbrace i\in\Z_+ : Z_i \leq k-1 \right \rbrace$ for $k\in\N$.

\begin{Cor}\label{cor:DecompositionRWCSFutureInfimum} Under $P^{\uparrow}_0$, the reversed excursions
\[\left\lbrace  Z^{(k)}_i, \ 0\leq i < R_{k+1}-R_k \right\rbrace:=\left\lbrace Z_{R_{k+1}-i}-k,  \ 0\leq i < R_{k+1}-R_k \right\rbrace, \quad k\in\Z_+\] are independent and distributed as $\{Z_{k} : 0\leq k < T\}$ under $P^{\widehat{\nu}}_{0}$.
	
\end{Cor} The law of the conditioned random walk stopped at a first hitting time is explicit.

\begin{Lem}\label{lem:RWCSPkilled} Let $n\in\Z_+$. Under $P^{\uparrow}_0$, $\{Z_k : 0\leq k \leq T_n\}$ has distribution $P_0\left( \cdot \mid T_n<T \right)$. \end{Lem}

\begin{proof}Let $k\in\Z_+$ and $x_0,\ldots,x_{k} \geq 0$. By \cite[Theorem 1]{bertoin_conditioning_1994}, 
\[P_0\left( T_n=k, Z_0=x_0,\ldots,Z_{k}=x_{k} \mid T_m<T \right) \underset{m \rightarrow \infty}{\longrightarrow} P_0^{\uparrow}\left(T_n=k, Z_0=x_0,\ldots,Z_{k}=x_{k}\right).\] It is thus sufficient to prove that for $m$ large enough, 
\[P_0\left( T_n=k, Z_0=x_0,\ldots,Z_{k}=x_{k} \mid T_m<T \right)=P_{0}\left(T_n=k, Z_0=x_0,\ldots,Z_{k}=x_{k} \mid T_n<T \right).\] We have $T_n<T_m$ whenever $m>n$, and $Z_{T_n}=n$ $P_0$-a.s.. The strong Markov property gives
\begin{align*}
  P_0&\left(T_n=k, Z_0=x_0,\ldots, T_m<T \right)=P_0\left(T_n=k, Z_0=x_0,\ldots, T_n<T \right)P_n\left(T_m<T \right).\end{align*} We conclude the proof by using the identity $P_0\left(T_m<T \right)=P_0\left(T_n<T \right)P_n\left(T_m<T \right)$.\end{proof}
  
We now deal with the conditioned random walk started at large values. For every $x\geq 0$, let $\phi_x$ be defined for every $y\in\R^\N$ by $\phi_x(y)=(y_i-x : i\in\N)$. We use the notation $f_\ast P$ for the pushforward measure of $P$ by the function $f$.
 
\begin{Lem}\label{lem:ConvergenceRWCSPLargeX}In the sense of weak convergence of finite-dimensional distributions,
\[(\phi_{x})_\ast P^\uparrow_x \underset{x \rightarrow \infty}{\Longrightarrow} P_0.\] 

\end{Lem}

\begin{proof}Let $k\in\Z_+$ and $x_0,\cdots,x_{k} \in \Z$. From \eqref{eqn:HTransform}, we have \[(\phi_{x})_\ast P^\uparrow_x(Z_0=x_0,\ldots,Z_k=x_k)=\frac{V(x+x_k)}{V(x)}P_x(k>T,Z_0=x+x_0,\ldots,Z_k=x+x_k).\] Up to choosing $x$ large enough, we can assume that $x_i+x\geq 0$ for $i\in\{ 0,\ldots,k \}$ and get  \[(\phi_{x})_\ast P^\uparrow_x(Z_0=x_0,\ldots,Z_k=x_k)=\frac{V(x+x_k)}{V(x)}P_0(Z_0=x_0,\ldots,Z_k=x_k).\] We now assume that $x_k\in\Z_+$ (the case $x_k\in\Z_-$ can be treated similarly). We have \[\frac{V(x+x_k)}{V(x)}=1+\frac{1}{V(x)}E_0\left(\sum_{k=0}^{T_0-1}{\mathbf{1}_{\{-x-x_k \leq Z_k< -x\}}}\right).\] By monotone convergence, since $T_0<\infty$ $P_0$-a.s., \[ E_0\left(\sum_{k=0}^{T_0-1}{\mathbf{1}_{\{-x-x_k \leq Z_k< -x\}}}\right)\leq E_0\left(\sum_{k=0}^{T_0-1}{\mathbf{1}_{\{Z_k< -x\}}}\right)\underset{x\rightarrow \infty}{\longrightarrow}0.\] By \eqref{eqn:RenewalFunction} and monotone convergence once again, since $E_0(T_0)=\infty$, $V(x)\rightarrow \infty$ as $x\rightarrow \infty$, which concludes the proof.\end{proof}

\bigskip

\noindent\textbf{Random walk with positive drift.} We consider random walks with positive drift conditioned to stay nonnegative. Let $\nu$ and $\{\nu_p : p\in \Z_+\}$ be upwards-skip-free probability measure, such that $\nu$ is centered, $m_p:=E^{\nu_p}_0(Z_1)>0$ and $\nu_p \Rightarrow \nu$ weakly as $p\rightarrow \infty$. The random walk conditioned to stay nonnegative is well defined in the usual sense, since $P^{\nu_p}_0(T=\infty)>0$ for $p\in \Z_+$. We denote its law by $P_x^{\nu_p \uparrow}$. It is also the h-transform of $P^{\nu_p}_x$ by the renewal function $V_p$ associated to the strict ascending ladder height process of $-Z$, which satisfies 
\[V_p(x)=\frac{P^{\nu_p}_x(T=\infty)}{P^{\nu_p}_0(T=\infty)}, \quad x\geq 0, \ p\in \Z_+.\] We let $T_{-n}:=\inf \left\lbrace k\in\N : Z_k\leq -n \right\rbrace$ for every $n\in\N$.

\begin{Lem}\label{lem:ConvergenceRWCSPDrift} For $x\in \Z_+$, in the sense of weak convergence of finite-dimensional distributions,
\[P^{\nu_p \uparrow}_x \underset{p \rightarrow +\infty}{\Longrightarrow} P^{\nu \uparrow}_x.\] 

\end{Lem}

\begin{proof}Let $x,k\in\Z_+$ and $x_0,\cdots,x_{k} \in \Z_+$. By \eqref{eqn:HTransform},
\begin{equation}\label{eqn:ConvergenceRWCSPDrift}
	P^{\nu_p \uparrow}_x(Z_0=x_0,\ldots,Z_k=x_k)=\frac{V_p(x_k)}{V_p(x)}P^{\nu_p}_x(k>T,Z_0=x_0,\ldots,Z_k=x_k).
\end{equation} Since $\nu_p \Rightarrow \nu$ weakly, we have
\[P^{\nu_p}_x(k>T,Z_0=x_0,\ldots,Z_k=x_k)\underset{p \rightarrow \infty}{\longrightarrow}P^{\nu}_x(k>T,Z_0=x_0,\ldots,Z_k=x_k).\] By \eqref{eqn:RenewalFunction}, for every $x,k\in\Z_+$,
\[V_p(x)=E^{\nu_p}_0\left(\sum_{k=0}^\infty{\mathbf{1}_{\{H_k\leq x\}}}\right)=E^{\nu_p}_0\left(\mathbf{1}_{\{T_{-x}\leq K\}}\sum_{k=0}^\infty{\mathbf{1}_{\{H_k\leq x\}}}\right)+E^{\nu_p}_0\left(\mathbf{1}_{\{T_{-x}> K\}}\sum_{k=0}^\infty{\mathbf{1}_{\{H_k\leq x\}}}\right).\] The first variable being measurable with respect to the $K$ first steps of $Z$, we get \[E^{\nu_p}_0\left(\mathbf{1}_{\{T_{-x}\leq K\}}\sum_{k=0}^\infty{\mathbf{1}_{\{H_k\leq x\}}}\right) \underset{p \rightarrow \infty}{\longrightarrow} E^{\nu}_0\left(\mathbf{1}_{\{T_{-x}\leq K\}}\sum_{k=0}^\infty{\mathbf{1}_{\{H_k\leq x\}}}\right).\] Since $H$ is the \emph{strict} ascending ladder height process of $-Z$ and $Z$ takes only integer values, 
\[E^{\nu_p}_0\left(\mathbf{1}_{\{T_{-x}> K\}}\sum_{k=0}^\infty{\mathbf{1}_{\{H_k\leq x\}}}\right)\leq (x+1)P^{\nu_p}_0\left(T_{-x}> K\right)\underset{p \rightarrow \infty}{\longrightarrow}(x+1)P^{\nu}_0\left(T_{-x}> K\right).\] As a consequence, $\limsup_{p\rightarrow\infty}\vert V_p(x) - V(x) \vert \leq 2(x+1)P^{\nu}_0\left(T_{-x}> K\right)$. Furthermore, $T_{-x}$ is finite $P^{\nu}_0$-a.s., so that $V_p(x)\rightarrow V(x)$ as $p\rightarrow \infty$ for every $x\geq 0$. Applying this to \eqref{eqn:ConvergenceRWCSPDrift} together with \eqref{eqn:HTransform} yields the expected result.\end{proof}

\subsection{Contour functions of random looptrees.}\label{sec:ContourFunctionRandom} Let $\nu$ be a centered upwards-skip-free probability measure on $\Z$ such that $\nu(0)=0$. We define the probability measures $\nuw$ and $\nub$ (with means $m_\circ$ and $m_\bullet$) by
\begin{equation}\label{eqn:DefinitionNuwNub}
	\nuw(k):=\nu(1)\left(1-\nu(1)\right)^k \quad \text{and} \quad \nub(k):=\frac{\nu(-k)}{1-\nu(1)}, \quad k\in\Z_+.  
\end{equation} The fact that $\nu$ is centered entails $m_\circ m_\bullet=1$, i.e.\ the pair $(\nuw,\nub)$ is critical.
 
\bigskip

\noindent\textbf{Finite looptrees.} We consider a random walk $\{\widehat{C}_k : k\in\Z_+\}$ with law $P^{\widehat{\nu}}_0$ (and $\widehat{\nu}:=\nu(-\cdot)$). We let $T:=\inf \{ k\in\Z_+ : \widehat{C}_k<0 \}$, and $C=\{C_k : 0\leq k < T\}:=\{\widehat{C}_{T-1-k} : 0\leq k < T\}$. 

\begin{Lem}\label{lem:FiniteTree}
The tree of components $\Tr_{C}=\Tree(\Lt_C)$ of the looptree $\Lt_{C}$ has law $\GWnn$.\end{Lem}

\begin{proof}The excursion $C$ satisfies a.s.\ the assumptions of Section \ref{sec:ContourFunction} and defines a looptree $\Lt_C$. By construction, the number of offspring of the root vertex in $\Tr_{C}$ is the number of excursions of $\widehat{C}$ above zero before $T$: $k_{\emptyset}(\Tr_{C})=\inf \{ k\in\Z_+ : \widehat{C}_{\sigma_k+1}<0 \}$, where $\sigma_0=0$ and for every $k\in\Z_+$, $\sigma_{k+1}=\inf \{ i > \sigma_{k} : \widehat{C}_i\leq 0 \}$. By the strong Markov property, $k_{\emptyset}(\Tr_{C})$ has geometric distribution with parameter $\widehat{\nu}(-1)$, which is exactly $\nuw$. The descendants of the children of the root are coded by the excursions $\{\widehat{C}_{\sigma_{k+1}-i} :0\leq i \leq \sigma_{k+1}-\sigma_{k}\}$ for $0\leq k < k_{\emptyset}(\Tr_{C})$ and are i.i.d.. Thus, we focus on the child of the root $v$ coded by the first of these excursions (i.e., the first child if $k_{\emptyset}(\Tr_{C})=1$, the second otherwise).

The number of offspring of $v$ is $k_{v}(\Tr_{C})=\widehat{C}_1$ (conditionally on $\{T>1\}$, i.e., the root vertex has at least one child). Its law is $\widehat{\nu}$ conditioned to take positive values, which is $\nub$. The descendants of the children of $v$ are coded by the excursions $\{\widehat{C}_{\sigma'_{k+1}-i-1} : 0\leq i < \sigma'_{k+1}-\sigma'_{k}\}$ for $0\leq k < k_{v}(\Tr_C)$ (where $\sigma'_{0}=1$ and $\sigma'_{k+1}=\inf \{ i\geq \sigma'_{k} : \widehat{C}_i<\widehat{C}_{\sigma'_k} \}$). These excursions are independent with the same law as $\{C_k : 0\leq k < T\}$, which concludes the argument.\end{proof}

\begin{Rk}\label{rk:Forest} We can extend Lemma \ref{lem:FiniteTree} to $C=\{\widehat{C}_{T-1-k} : 0\leq k < T\}$ where $\widehat{C}$ has distribution $P^{\widehat{\nu}}_x$ with $x>0$ (using the construction of Section \ref{sec:ContourFunction}). By decomposing $\widehat{C}$ into its excursions above its infimum (as in \eqref{eqn:ExcursionIntervals}), we get a forest $\mathbf{F}_C$ of $x+1$ independent looptrees, whose trees of components have distribution $\GWnn$ (see Figure \ref{fig:FiniteForest} for an illustration).\end{Rk}

\begin{figure}[!ht]
\centering
\includegraphics[scale=1.5]{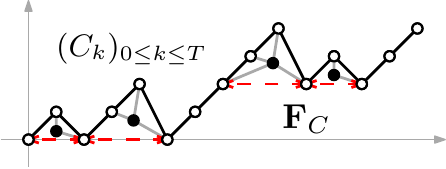}
\caption{The construction of the forest $\mathbf{F}_C$ from the excursion $C$ with $C_{T-1}=4$.}
\label{fig:FiniteForest}
\end{figure}

\bigskip

\noindent\textbf{Infinite looptrees.} We first consider a random walk $C=\{C_k : k\in\Z_+\}$ with law $P^\nu_0$. 

\begin{Prop}\label{prp:ComponentTree}
The tree of components $\Tr_C=\Tree(\Lt_C)$ of the looptree $\Lt_C$ has law $\GWnn^{(\infty,l)}$.\end{Prop} 

\begin{proof}The function $C$ satisfies a.s.\ the assumptions of Section \ref{sec:ContourFunction}. We denote by $\{s_k : k\in\Z_+\}$ the a.s.\ unique spine of $\Tr_C$. Recall from (\ref{eqn:ExcursionIntervals}) the definition of the excursions endpoints $\{\tau_k : k\in \Z_+\}$. The (white) vertices of the spine $\{s_{2k} : k\in\Z_+\}$ are also $\{p_C(\tau_k) : k\in\Z_+\}$. For every $k\in\Z_+$, let $\Tr_k$ be the sub-tree of $\Tr_C$ containing all the offspring of $s_{2k}$ that are not offspring of $s_{2k+2}$ (with the convention that $s_{2k}$ and $s_{2k+2}$ belong to $\Tr_k$). The tree $\Tr_k$ is coded by $\{C_{\tau_k + i}-C_{\tau_k} : 0\leq i \leq \tau_{k+1}-\tau_k \}$, see Figure \ref{fig:SpineLooptree}. By the strong Markov property, the trees $\{\Tr_k : k\in\Z_+\}$ are i.i.d.\ and it suffices to determine the law of $\Tr_0$.

The vertex $s_1$ is the unique black vertex of $\Tr_0$ that belongs to the spine of $\Tr_C$. Its number of offspring read $k_{s_{1}}=k_{s_{1}}(\Tr_C)=k_{s_{1}}(\Tr_0)=C_{\tau_{1}-1}-C_{\tau_{1}}$. Moreover, if $k^{(l)}_{s_{1}}$ (resp.\ $k^{(r)}_{s_{1}}$) is the number of offspring of $s_{1}$ on the left (resp.\ right) of the spine, we have $k^{(l)}_{s_{1}}=-C_{\tau_{1}}-1$ and $k^{(r)}_{s_{1}}=C_{\tau_{1}-1}$. Thus, the position of the child $s_{2}=p_C(\tau_{1})$ of $s_{1}$ that belongs to the spine among its $C_{\tau_{1}-1}-C_{\tau_{1}}$ children is $-C_{\tau_{1}}$. By Lemma \ref{lem:RandomWalkLastJumpTimeReverse},
\[P\left(k_{s_1}=j \right)=\frac{j\nu(-j)}{\nu(1)}=\bar{\nu}_{\bullet}(j), \quad j\in\N,\] and conditionally on $k_{s_1}$, the rank of $s_{2}$ is uniform among $\{1,\ldots,k_{s_1}\}$. We now work conditionally on $(C_{\tau_{1}}-1,C_{\tau_{1}})$, and let $\{v_j : 0\leq j \leq k_{s_{1}}\}$ be the neighbours of $s_{1}$ in counterclockwise order, $v_0$ being the root vertex of $\Tr_0$. Then, the descendants of $\{v_j : 0\leq j \leq k^{(r)}_{s_{1}}\}$ are the trees of components of the finite forest coded by $\{C_{i} : 0\leq i < \tau_{1}\}$ (see Figure \ref{fig:SpineLooptree}). By Lemma \ref{lem:RandomWalkLastJumpTimeReverse}, conditionally on $C_{\tau_1-1}$, the reversed excursion $\{C_{\tau_1-i-1} : 0\leq i < \tau_{1}\}$ has distribution $P^{\widehat{\nu}}_{C_{\tau_1-1}}$ and by Remark \ref{rk:Forest}, the trees of components form a forest of $C_{\tau_1-1}+1$ independent trees with distribution $\GWnn$, that are grafted on the right of the vertices $\{v_j : 0\leq j \leq k^{(r)}_{s_{1}}\}$. By construction, children of $s_1$ on the left of the spine have no offspring. It remains to identify the offspring distribution of the white vertex of the spine in $\Tr_0$, i.e.\ the root vertex. It has one (black) child $s_1$ on the spine, which is its leftmost offspring, and a tree with distribution $\GWnn$ grafted on the right of it. As a consequence, 
\[P\left( k_{\emptyset}(\Tr_C)=k \right)=\nuw(k-1), \quad k\in\N.\] By standard properties of the geometric distribution, this is the law of a uniform variable on $\{1,\ldots,X\}$, with $X$ of law $\bar{\nu}_\circ$. We get Kesten's multi-type tree with pruning on the left.\end{proof}

\begin{figure}[!ht]
\centering
\includegraphics[scale=1.]{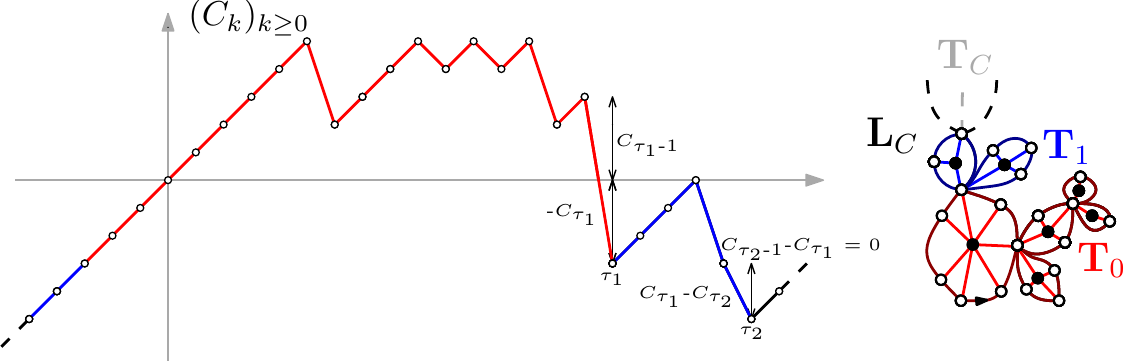}
\caption{The decomposition of the tree $\Tr_C$.}
\label{fig:SpineLooptree}
\end{figure}

Lastly, we consider a random walk $C=\{C_k : k\in\Z_+\}$ with law $P^\nu_0$, together with an independent process $C'=\{C'_k : k\in\Z_+\}$ with law $P^{\nu \uparrow}_0$. 
\begin{Prop}\label{prp:ComponentTreeBiInfinite}The tree of components $\Tr_{C,C'}=\Tree(\Lt_{C,C'})$ of $\Lt_{C,C'}$ has law $\GWnn^{(\infty)}$.\end{Prop} 

\begin{proof} The processes $C$ and $C'$ satisfy a.s. the assumptions of Section \ref{sec:ContourFunction}. There, we defined $\Lt_{C,C'}$ as the gluing along their boundaries of $\Lt_{C}$ and the infinite forest $\mathbf{F}_{C'}$. By Proposition \ref{prp:ComponentTree}, $\Tr_C=\Tree(\Lt_C)$ has distribution $\GWnn^{(\infty,l)}$. Recall from Section \ref{sec:ContourFunction} that the looptrees $\{\Lt_k : k\in \Z_+\}$ defining $\mathbf{F}_{C'}$ are coded by the excursions $\{C'_{R_{k}+i+1}-k : 0\leq i \leq R_{k+1}-R_k-1\}$. By Lemma \ref{cor:DecompositionRWCSFutureInfimum}, the time-reverse of these excursions are independent and distributed as $(Z_{k})_{0\leq k \leq T-1}$ under $P^{\widehat{\nu}}_{0}$. By Lemma \ref{lem:FiniteTree}, $\Tr_k:=\Tree(\Lt_k)$ has distribution $\GWnn$. So vertices of $\Tr_{C,C'}$ have independent number of offspring, and offspring distribution $\nuw$ and $\nub$ out of the spine. 

By Proposition \ref{prp:ComponentTree}, black vertices of the spine have offspring distribution $\bar{\nu}_\bullet$, and a unique child in the spine with uniform rank conditionally on the number of offspring. It remains to identify the offspring distribution of white vertices of the spine, and thus of the root vertex. By construction, we have $k_{\emptyset}\left(\Tr_{C,C'}\right)=k_{\emptyset}\left(  \Tr_0 \right)+k_{\emptyset}\left( \Tr_C \right)$, where the variables on the right hand side are independent with respective distribution $\nuw(\cdot -1)$ and $\nuw$. Then,	
	\[P\left(k_{\emptyset}\left(\Tr_{C,C'}\right)=k \right)=\sum_{i=1}^{k}{\nuw(i-1)\nuw(k-i)}=\bar{\nu}_\circ(k), \quad k\in\N.\] The child of the root that belongs to the spine of $\Tr_{C,C'}$ is the leftmost child of the root in $\Tr_C$. Its rank among the $k_{\emptyset}\left(\Tr_{C,C'}\right)$ children of the root in $\Tr_{C,C'}$ is $k_{\emptyset}\left(  \Tr_0 \right)+1$. Furthermore,
	\[P\left( k_{\emptyset}\left(  \Tr_0 \right)+1=i \mid k_{\emptyset}\left(  \Tr_{C,C'} \right)=k \right)=\nuw(i-1)\nuw(k-i)\frac{m_{\circ}}{k\nuw(k)}=\frac{1}{k},\quad k\in\N, \ 1\leq i \leq k,\] so that this rank is uniform among the offspring. We obtain Kesten's multi-type tree.
\end{proof}

\section{Decomposition of the $\normalfont{\UIHPT}$}\label{sec:DecompositionUIHPT}

In this section, we introduce a decomposition of the $\UIHPT$ along a percolation interface and prove Theorem \ref{thm:DecompositionTheorem}. The idea of this decomposition first appears in~\cite[Section 1.7.3]{duplantier_liouville_2014}, where it served as a discrete intuition for a continuous model (see Section \ref{sec:ScalingLimits} for details).

\subsection{Exploration process}\label{sec:ExplorationProcess}

We consider a Bernoulli percolation model with parameter $p$ on the $\UIHPT$, conditionally on the ``Black-White" boundary condition of Figure \ref{fig:InitialColouring}. The decomposition arises from the exploration of the percolation interface between the open and closed clusters of the boundary. Our approach is based on a peeling process introduced in \cite{angel_scaling_2004}, notably to compute the critical threshold. Although we will only use this peeling process at criticality in the remainder of this section, we define it for any $p\in(0,1)$ in view of forthcoming applications.

\begin{Alg}{(\cite{angel_scaling_2004})}\label{alg:AngelPeelingProcess} Let $p\in(0,1)$, and consider a percolated $\UIHPT$ with distribution $\Pb_p$ and a ``Black-White" boundary condition.

\begin{itemize}
\item Reveal the face incident to the edge of the boundary whose endpoints have different colour (the ``Black-White" edge).
\item Repeat the algorithm on the $\UIHPT$ given by the unique infinite connected component of the map deprived of the revealed face.
\end{itemize}

\end{Alg}

This peeling process is well defined in the sense that the pattern of the boundary is preserved, and the spatial Markov property implies that its steps are i.i.d.. We now introduce a sequence of random variables describing the evolution of the peeling process. Let $\Ec_k$, $\Rc_{l,k}$, $\Rc_{r,k}$ and $c_k$ be the number of exposed edges, swallowed edges on the left and right, and colour of the revealed vertex (if any, or a cemetery state otherwise) at step $k$ of the peeling process.
\begin{Def}\label{def:ExplorationProcess}
For every $k\in\N$, let $Y_k=(Y^{(1)}_k,Y^{(2)}_k):=\left(\mathbf{1}_{\{c_k=1\}}-\Rc_{l,k},\mathbf{1}_{\{c_k=0\}}-\Rc_{r,k}\right)$. The exploration process $X=\{X_k=(X^{(1)}_k,X^{(2)}_k) : k\in\Z_+\}$ is defined by
\[X_0=(0,0) \quad \text{and} \quad X_k:=\sum_{i=1}^k{Y_i}, \quad k\in\N.\] \end{Def}

By the properties established in Section \ref{sec:RPM}, the variables $\{Y_k : k\in\N\}$ are independent and distributed as $Y=\left( Y^{(1)},Y^{(2)} \right)$ such that
\begin{equation}\label{eqn:LawY}
Y=\left\lbrace
\begin{array}{cccc}
(1,0) & \mbox{with probability} & 2p/3 &  \\
(0,1) & \mbox{with probability} & 2(1-p)/3 &  \\
(-k,0) & \mbox{with probability} & q_k & (k\in\N)\\
(0,-k) & \mbox{with probability} & q_k & (k\in\N)\\
\end{array}\right..
\end{equation} Let $\mu^0_p$ be the law of $ Y^{(1)}$, so that $Y^{(2)}$ has law $\mu^0_{1-p}$. (This defines a probability measure since $\sum_{k\in\N}{q_k}=1/6$.) Note that $Y^{(1)}=0$ or $Y^{(2)}=0$ but not both a.s., and $\mu^0_{1/2}$ is centered (since $\sum_{k\in\N}{kq_k}=1/3$). We call the random walk $X$ the \textit{exploration process}, as it fully describes Algorithm \ref{alg:AngelPeelingProcess}. We now extract from $X$ information on the percolation clusters. Let $\sigma^B_0=\sigma^W_0=0$ and for every $k\in\N$,
\begin{equation}\label{eqn:BlackWhiteTimes} \sigma^B_k:=\inf\left\lbrace i\geq \sigma^B_{k-1} : X^{(1)}_i \neq X^{(1)}_{\sigma^B_{k-1}} \right\rbrace \quad \text{and} \quad \sigma^W_k:=\inf\left\lbrace i\geq \sigma^W_{k-1} : X^{(2)}_i \neq X^{(2)}_{\sigma^W_{k-1}} \right\rbrace.\end{equation}These stopping times are a.s.\ finite. We also let \begin{equation}\label{eqn:DefZ}
	z_k:=\mathbf{1}_{\left\lbrace Y^{(1)}_k\neq 0 \right\rbrace}=\mathbf{1}_{\left\lbrace Y^{(2)}_k=0 \right\rbrace},\quad k\in\N,
\end{equation} In a word, $\{z_k : k\in\N\}$ is the sequence of colours of the third vertex of the faces revealed by the exploration. The processes $B=\{B_k : k\in\Z_+\}$ and $W=\{W_k : k\in\Z_+\}$ are defined by
\[B_k:=X^{(1)}_{\sigma^B_k}\quad \text{and} \quad W_k:=X^{(2)}_{\sigma^W_k}, \quad k\in\Z_+.\] 

\begin{Lem}\label{lem:LawBW} Let $p\in(0,1)$. Under $\Pb_p$, $B$ and $W$ are independent random walks started at $0$ with step distribution \[\mu_p:=\frac{\mu^0_p(\cdot \cap \Z^*)}{\mu^0_p(\Z^*)}.\] Moreover, $\{z_k : k\in\N\}$ are independent with Bernoulli distribution of parameter $g(p):=2p/3+1/6$, and independent of $B$ and $W$.
\end{Lem}

\begin{proof}As we noticed, for every $k\in\Z_+$, $Y^{(1)}_k=0$ or $Y^{(2)}_k=0$ a.s.. Thus, the sequences $\{\sigma^B_k : k\in\N\}$ and $\{\sigma^W_k : k\in\N\}$ induce a partition of $\N$. Moreover, we have \[B_k=\sum_{i=0}^k{Y^{(1)}_{\sigma^B_i}} \quad \text{and} \quad W_k=\sum_{i=0}^k{Y^{(2)}_{\sigma^W_i}}, \quad k\in\Z_+.\] The variables $\{Y_k : k\in\N\}$ being independent, $B$ and $W$ are independent. We also have	  
	  \[\sigma^B_k:=\inf\left\lbrace i\geq \sigma^B_{k-1} : Y^{(1)}_k \neq 0 \right\rbrace \quad \text{ and } \quad \sigma^W_k:=\inf\left\lbrace i\geq \sigma^W_{k-1} : Y^{(2)}_k \neq 0 \right\rbrace, \quad k\in\N.\] The strong Markov property applied at these times entails the first assertion. The distribution of the variables $\{z_k : k\in\N\}$ follows from the definition of the exploration process and the identity $\mu^0_p(0)=2p/3+1/6$. Finally, for every $k\in\N$, $\sigma^B_k$ (resp.\ $\sigma^W_k$) is independent of $Y_{\sigma^B_k}$ (resp.\ $Y_{\sigma^W_k}$) so that $\{z_k : k\in\N\}$ is independent of $B$ and $W$.\end{proof}
	  
	  Note that for $p=p_c=1/2$, $\mu:=\mu_{p_c}$ is centered and $g(p_c)=1/2$, while $\mu_p$ has positive mean if $p>p_c$. In the remainder of this section, we assume that $p=p_c$ and work under $\Pb$.

\subsection{Percolation hulls and necklace}\label{sec:PercolationClusters}

We now describe the percolation hulls $\Hb$ and $\Hw$ of the origin and the target of the root at criticality. Recall from (\ref{eqn:DefinitionMuwMub}) the definition of the probability measures $\muw$ and $\mub$.

\begin{Prop}\label{prp:BoundaryPercolationHull} The trees of components $\Treeb(\Hb)$ and $\Treeb(\Hw)$ are independent with respective distribution $\GWmm^{(\infty,l)}$ and $\GWmm^{(\infty,r)}$. 

Moreover, conditionally on $\Treeb(\Hb)$ and $\Treeb(\Hw)$, the irreducible components $\{M^b_v : v \in \Treeb(\Hb)_\bullet\}$ and $\{M^w_v : v \in \Treeb(\Hw)_\bullet\}$ are independent critically percolated Boltzmann triangulations with a simple boundary and respective distribution $\mathbf{W}_{\deg(u)}$.\end{Prop} 

\begin{Rk} Galton-Watson trees conditioned to survive being a.s.\ locally finite and one-ended, $\Scoop(\Hb)$ and $\Scoop(\Hw)$ are infinite looptrees. In particular, $\Hb$ or $\Hw$ may have edges whose both sides are incident to their root face, corresponding to a loop of perimeter $2$ in $\Scoop(\Hb)$ or $\Scoop(\Hw)$ that is filled in with the Boltzmann triangulation with a simple boundary consisting of a single edge.
\end{Rk}

\begin{proof}We first deal with the open hull $\Hb$. The equivalence relation $\sim$ is defined by applying (\ref{eqn:EquivC}) with $B$. For every $k\in\Z_+$, let $\Lt_k$ be the quotient space of $(-\infty,k]\cap \Z$ by the restriction of $\sim$ to this set (with the embedding convention of Section \ref{sec:ContourFunction}). We also let $\mathbf{S}_k$ be the part of $\Scoop(\Hb)$ discovered at step $\sigma^B_k$ of the peeling process.

We now prove by induction that $\mathbf{S}_k$ is a map isomorphic to $\Lt_k$ for every $k\in\Z_+$. This is clear for $k=0$, since the initial open cluster is isomorphic to $\mathbb{N}$. We assume that this holds for $k\in\N$, and denote by $p_k$ the canonical projection on $\Lt_k$. The exploration steps between $\sigma^B_k$ and $\sigma^B_{k+1}$ reveal the face incident to $p_k(k)$ and the leftmost white vertex. They leave invariant the open cluster, so we restrict our attention to the step $\sigma^B_{k+1}$ at which two cases are likely.

\begin{enumerate}
	\item An inner open vertex is discovered $(Y^{(1)}_{\sigma^B_{k+1}}=1)$. Then, $\mathbf{S}_k$ is isomorphic to $\Lt_k$ plus an extra vertex in its external face connected only to $p_k(k)$. On the other hand, $B_{k+1}=B_k+1$ so that $\Lt_{k+1}$ is isomorphic to $\mathbf{S}_k$.
	\item The third vertex of the revealed triangle is on the (left) boundary and $l\in\N$ edges are swallowed $(Y^{(1)}_{\sigma^B_{k+1}}=-l)$. Then, $\mathbf{S}_k$ is isomorphic to $\Lt_k$ plus an edge between $p_k(k)$ and $l$-th vertex after $p_k(k)$ in left contour order on the boundary of $\Lt_k$. Since $B_{k+1}=B_k-l$, $\mathbf{S}_k$ is isomorphic to $\Lt_{k+1}$.
\end{enumerate} In the second case, by the spatial Markov property, the loop of perimeter $l+1$ added to $\Lt_k$ is filled in with an independent percolated triangulation with a simple boundary having distribution $\mathbf{W}_{l+1}$. This is the irreducible component $M^\bullet_v$ associated to this loop in $\Scoop(\Hb)$. The peeling process follows the right boundary of $\Hb$. Since $\liminf_{k \rightarrow +\infty}{B_k}=-\infty$ a.s., the left and right boundaries of the hull intersect infinitely many times during the exploration. This ensures that the whole hull $\Hb$ is revealed by the peeling process (i.e.\ that $\B_R(\Hb)$ is eventually revealed for every $R\geq 0$). Moreover, the sequence $\{\Lt_k : k\in\N\}$ is a consistent exhaustion of the looptree $\Lt_B$. Thus, the scooped-out boundary $\Scoop(\Hb)$ is isomorphic to $\Lt_B$. By Proposition \ref{prp:ComponentTree}, the tree of components $\Tr_B=\Treeb(\Hb)$ has distribution $\GWmm^{(\infty,l)}$.

The same argument shows that $\Scoop(\Hw)$ is isomorphic to $\Lt_W$ (up to a reflection and a suitable rooting convention), and the independence of $B$ and $W$ concludes the proof.\end{proof}

\begin{Rk}We proved that $\Hb$ and $\Hw$ have the same law up to a reflection, which is clear since $p_c=1/2$. Note that $\mub$ being supported on $\N$, there is no black leaf in the tree of components a.s.\ (this would correspond to a self-loop in the $\UIHPT$, which is not allowed under our 2-connectedness assumption).
\end{Rk}

Let us explain how the hulls are connected in the $\UIHPT$. We define a planar map with an infinite simple boundary $\mathbf{N}$ as follows. Let $\{c_i : i\in \N\}$ be the corners of the right boundary of $\Hb$ listed in contour order, and similarly for $\{c'_i : i\in \N\}$ with the left boundary of $\Hw$. Then, let $\mathbf{N}$ be the planar map with vertex set $\{c_i, c'_i : i\in \N\}$, such that two vertices are neighbours iff the associated corners are connected by an edge in the $\UIHPT$. (Loosely speaking, we consider the sub-map of the $\UIHPT$ generated by the right boundary of $\Hb$ and the left boundary of $\Hw$, but we split the pinch-points of these boundaries.)

\begin{Prop}\label{prp:UniformNecklace} The planar map $\mathbf{N}$ has distribution $\mathsf{UN}(\infty,\infty)$. Otherwise said, $\Hb$ and $\Hw$ are glued along an independent uniform infinite necklace.\end{Prop}

\begin{Rk} Due to the ``Black-White" boundary condition, Proposition \ref{prp:UniformNecklace} ensures that the triangulation generated by $\Hb$ and $\Hw$ in the $\UIHPT$ (i.e.\ the map revealed by the peeling process) is $\Psi_{\mathbf{N}}(\Hb,\Hw)$. Note that the edges of the uniform infinite necklace are exactly the edges that are peeled.\end{Rk}

\begin{proof}For every $k\in\Z_+$ such that $Y^{(1)}_k\neq 0$, the third vertex of the revealed face at step $k$ of the peeling process is open, and defines the $k$-th corner of $\partial\Hb$ in right contour order. In such a situation, there is an edge between this corner and the last (closed) corner of $\partial\Hw$ that has been discovered. The converse occurs when $Y^{(2)}_k\neq 0$. By Lemma \ref{lem:LawBW}, the variables
\[z_k:=\mathbf{1}_{\left\lbrace Y^{(2)}_k=0 \right\rbrace}=\mathbf{1}_{\left\lbrace Y^{(1)}_k\neq 0 \right\rbrace}, \quad  k\in\N,\] are independent with Bernoulli distribution of parameter $1/2$, and independent of $B$ and $W$. We obtain the uniform infinite necklace.\end{proof}

\subsection{Proof of the decomposition result}

\begin{proof}[Proof of Theorem \ref{thm:DecompositionTheorem}]
	
The proof is based on Propositions \ref{prp:BoundaryPercolationHull} and \ref{prp:UniformNecklace}. However, it remains to show that the percolation hulls $\Hb$ and $\Hw$, and the infinite necklace $\mathbf{N}$ cover the entire map, or in other words that the peeling process discovers the whole $\UIHPT$. 

For every $n\in \N$, let $M_n$ be the map revealed at step $n$ of the peeling process, and denote by $M_{\infty}$ the underlying $\UIHPT$ (with origin vertex $\rho$). We denote by $\{\tau^{\bullet}_k : k\in\Z_+\}$ and $\{\tau^{\circ}_k : k\in\Z_+\}$ the endpoints of the excursions intervals of $B$ and $W$ above their infimums, defined as in (\ref{eqn:ExcursionIntervals}), and set 	
	\[\phi(n):=\sigma^B_{\tau^{\bullet}_n} \vee \sigma^W_{\tau^{\circ}_n}, \quad n\in \N.\] We have $\phi(n)<\infty$ a.s.\ for every $n\in\N$. We consider the sequence of sub-maps of $M_\infty$ given by $\{M_{\phi(n)} : n\in \N\}$. For every $n\in \N$, we let $d_n$ stand for the graph distance on $M_{\phi(n)}$ and denote by $\partial M_{\phi(n)}$ the boundary of $M_{\phi(n)}$ as a sub-map of $M_\infty$. More precisely,
	\begin{equation}\label{eqn:DefBoundary}
		\partial M_{\phi(n)}:=\left\lbrace v\in M_{\phi(n)} : \exists \ u\in M_\infty\backslash M_{\phi(n)} \text{ s.t. } u\sim v \right\rbrace.
	\end{equation} Let $\{v^{\bullet}_k : k\in\Z_+\}$ and $\{v^{\circ}_k : k\in\Z_+\}$ be the white vertices of the spine in $\Tr_B$ and $\Tr_W$, seen as vertices of $\Lt_B=\Scoop(\Hb)$ and $\Lt_W=\Scoop(\Hw)$ (and thus of $\Hb$ and $\Hw$). Namely, 
\[v^{\bullet}_k:=p_B\left( B_{\tau^{\bullet}_k} \right) \quad \text{and} \quad v^{\circ}_k:=p_W\left( W_{\tau^{\circ}_k} \right), \quad k\in\Z_+.\] Note that the vertices $v^{\bullet}_0,\ldots,v^{\bullet}_n$ and $v^{\circ}_0,\ldots,v^{\circ}_n$ can be identified in $M_{\phi(n)}$, since $\{\tau^{\bullet}_k : k\in\Z_+\}$ and $\{\tau^{\circ}_k : k\in\Z_+\}$ are stopping times in the filtration of the exploration process. Moreover, these vertices are cut-points: they disconnect the origin from infinity in $\Lt_B$ (resp.\ $\Lt_W$). We now define an equivalence relation $\approx$ on the set $V(M_{\phi(n)})$ of vertices of $M_{\phi(n)}$ as follows:
\begin{center}For every $k\in\{0,\ldots,n\}$ and every $v\in V(\Hb\cap M_{\phi(n)})$, $v \approx v^\bullet_k$ iff there exists a geodesic path from $v$ to the origin of $\Hb$ that contains $v^\bullet_k$, but does not contain $v^\bullet_{k+1}$ if $k<n$.
\end{center} We define symmetric identifications on $V(\Hw\cap M_{\phi(n)})$. Roughly speaking, for every $k\in\{0,\ldots,n-1\}$, the vertices of $\Hb$ between $v^\bullet_k$ and $v^\bullet_{k+1}$ (excluded) are identified to $v^\bullet_k$, and all the vertices of $\Hb$ above $v^\bullet_n$ are identified to $v^\bullet_n$ (and similarly in $\Hw$). We denote the quotient map $M_{\phi(n)}/\approx$ by $M'_{\phi(n)}$ (the root edge of $M'_{\phi(n)}$ is the root edge of $M_\infty$). The graph distance on $M'_{\phi(n)}$ is denoted by $d'_n$. The family $\{M'_{\phi(n)} : n\in \N \}$ is a consistent sequence of locally finite maps with origin $v^\bullet_0=\rho$. Moreover, for every $n\in\N$, the boundary of $M'_{\phi(n)}$ in $M'_{\phi(n+1)}$ is $\{v^{\bullet}_n,v^{\circ}_n\}$. Thus, the sequences $\{d'_n\left( \rho,v^{\circ}_n \right) :n\in\N\}$ and $\{d'_n\left( \rho,v^{\bullet}_n \right) : n\in\N\}$ are non-decreasing and diverge a.s.. By definition of $\approx$ and since we discover the finite regions swallowed by the peeling process, the representatives of $\partial M_{\phi(n)}$ in $M'_{\phi(n)}$ are $v^{\bullet}_n$ and $v^{\circ}_n$. As a consequence,
\begin{equation}\label{eqn:DistanceBoundaryConvergence}
	d_{n}\left( \rho,\partial M_{\phi(n)} \right)\underset{n \rightarrow \infty}{\longrightarrow}\infty \quad \text{a.s.}.
\end{equation} This implies that a.s.\ for every $R\in \Z_+$, the ball of radius $R$ of the $\UIHPT$ $\mathbf{B}_R(M_{\infty})$ is contained in $M_{\phi(n)}$ for $n$ large enough, and concludes the argument.\end{proof}

\section{The incipient infinite cluster of the $\normalfont{\UIHPT}$}\label{sec:IIC}

The goal of this section is to introduce the $\IIC$ probability measure, and prove the convergence and decomposition result of Theorem \ref{thm:IICTheorem}.

\subsection{Exploration process}\label{sec:ExplorationProcessIIC}

We consider a Bernoulli percolation model with parameter $p$ on the $\UIHPT$, conditionally on the ``White-Black-White" boundary condition of Figure \ref{fig:InitialColouringIIC}. From now on, we assume that $p\in [p_c,1)$. As in Section \ref{sec:DecompositionUIHPT} we use a peeling process of the $\UIHPT$, which combines two versions of Algorithm \ref{alg:AngelPeelingProcess}.

\begin{Alg}\label{alg:AlgorithmIIC}
Let $p\in [p_c,1)$ and consider a percolated $\UIHPT$ with distribution $\Pb_p$ and ``White-Black-White" boundary condition. Let $n,m\in\N$ such that $n\geq m$. The first part of the algorithm is called the \textit{left peeling}.

\begin{enumerate}
	\item \textbf{Left peeling.} While the finite open segment on the boundary has size less than $n+1$:
	\begin{itemize}
\item Reveal the face incident to the (leftmost) ``White-Black" edge on the boundary. If there is no such edge, reveal the face incident to the leftmost exposed edge at the previous step.
\item Repeat the operation on the unique infinite connected component of the map deprived of the revealed face.
\end{itemize} When the finite open segment on the boundary has size larger than $n+1$, we start a second part which we call the \textit{right peeling}.

	\item \textbf{Right peeling.} While the finite open segment on the boundary of the map has size greater than $n+1-m$:
	\begin{itemize}
\item Reveal the face incident to the (rightmost) ``Black-White" edge on the boundary. If there is no such edge, reveal the face incident to the leftmost exposed edge at the previous step.
\item Repeat the operation on the unique infinite connected component of the map deprived of the revealed face.
\end{itemize}

The algorithm ends when the left and right peelings are completed.

\end{enumerate}

\end{Alg}

\begin{Rk}By definition, the left and right peelings stop when the length of the open segment reaches a given value. However, it is convenient to define both peeling processes \emph{continued forever}. We systematically consider such processes, and use the terminology \textit{stopped peeling process} otherwise. Nevertheless, the right peeling is defined on the event that the left peeling ends, i.e.\ that the open segment on the boundary reaches size $n+1$.
\end{Rk}

Intuitively, the left and right peelings explore the percolation interface between the open cluster of the origin and the closed clusters of its left and right neighbours on the boundary. Note that the peeling processes are still defined when the open segment on the boundary is swallowed, although they do not follow any percolation interface in such a situation. 
 
As for Algorithm \ref{alg:AngelPeelingProcess}, by the spatial Markov property, Algorithm \ref{alg:AlgorithmIIC} is well defined and has i.i.d.\ steps. We use the notation of Section \ref{sec:ExplorationProcess} for the number of exposed edges, swallowed edges on the left and right, and colour of the revealed vertex (if any) at step $k$ of the left and right peeling processes. We use the exponents $l$ and $r$ to distinguish the quantities concerning the left and right peelings.

The peeling processes are fully described by the associated exploration processes $X^{(l)}=\{X^{(l)}_k=(X^{(l,1)}_k,X^{(l,2)}_k) : k\in\Z_+\}$ and $X^{(r)}=\{X^{(r)}_k=(X^{(r,1)}_k,X^{(r,2)}_k) : k\in\Z_+\}$, defined by 
 \[X^{(l)}_0=X^{(r)}_0=(0,0) \quad \text{and} \quad X^{(l)}_k:=\sum_{i=1}^k{Y^{(l)}_i}, \ X^{(r)}_k:=\sum_{i=1}^k{Y^{(r)}_i},\quad k\in\N,\] where for every $k\in\N$, \[Y^{(l)}_k:=\left(\mathbf{1}_{\left\lbrace c^{(l)}_k=1\right\rbrace}-\Rc^{(l)}_{l,k},\mathbf{1}_{\left\lbrace c^{(l)}_k=0\right\rbrace}-\Rc^{(l)}_{r,k}\right) \text{ and }  Y^{(r)}_k:=\left(\mathbf{1}_{\left\lbrace c^{(r)}_k=1\right\rbrace}-\Rc^{(r)}_{l,k},\mathbf{1}_{\left\lbrace c^{(r)}_k=0\right\rbrace}-\Rc^{(r)}_{r,k}\right).\] The exploration processes $X^{(l)}$ and $X^{(r)}$ have respective lifetimes
\begin{equation}\label{eqn:Sigma}
   \sigma^{(l)}_n:=\inf \left\lbrace k\in\Z_+ : X^{(l,1)}_k\geq n \right\rbrace \quad \text{and} \quad \sigma^{(r)}_m:=\inf \left\lbrace k\in\Z_+ : X^{(r,1)}_k \leq -m \right\rbrace.
\end{equation} The right peeling process is defined on the event $\{\sigma^{(l)}_n<\infty\}$. However, the assumption $p\geq p_c$ guarantees that a.s.\ for every $n\in\N$, $\sigma^{(l)}_n<\infty$ and the left peeling ends. On the contrary, when $p>p_c$, with positive probability $\sigma^{(r)}_m=\infty$ and the right peeling does not end.

We now extract information on the percolation clusters, and introduce the processes $B^{(l)}=\{B^{(l)}_k : k\in\Z_+\}$, $W^{(l)}=\{W^{(l)}_k : k\in\Z_+\}$, $B^{(r)}=\{B^{(r)}_k : k\in\Z_+\}$ and $W^{(r)}=\{W^{(r)}_k : k\in\Z_+\}$ defined by
\[\left(B^{(l)}_k,W^{(l)}_k\right):=\left(X^{(l,1)}_{\sigma^{(B,l)}_k},X^{(l,2)}_{\sigma^{(W,l)}_k}\right) \quad \text{and} \quad \left(B^{(r)}_k,W^{(r)}_k\right):=\left(X^{(r,1)}_{\sigma^{(B,r)}_k},X^{(r,2)}_{\sigma^{(W,r)}_k}\right), \quad k\in\Z_+,\] where we use the same definitions as in (\ref{eqn:BlackWhiteTimes}) for the stopping times. We define as in (\ref{eqn:DefZ}) the random variables $\{z^{(l)}_k : k\in\N\}$ and $\{z^{(r)}_k : k\in\N\}$. The exploration process $X^{(l)}$ is measurable with respect to $B^{(l)}$, $W^{(l)}$ and $z^{(l)}$ (and the same holds when replacing $l$ by $r$). The lifetimes of the processes $B^{(l)}$ and $B^{(r)}$ are defined by
\begin{equation}\label{eqn:TTimeB}
  T^{(B,l)}_n:=\inf \left\lbrace k\in\Z_+ : B^{(l)}_k \geq n \right\rbrace \quad \text{and} \quad T^{(B,r)}_m:=\inf \left\lbrace k\in\Z_+ : B^{(r)}_k \leq -m \right\rbrace,
\end{equation} while the lifetimes of $W^{(l)}$ and $W^{(r)}$ read
\[T^{(W,l)}_n:=\sigma^{(l)}_n-T^{(B,l)}_n \quad \text{and} \quad T^{(W,r)}_m:=\sigma^{(r)}_m-T^{(B,r)}_m.\] When $p>p_c$, $\sigma^{(r)}_m=\infty$ with positive probability, in which case $T^{(B,r)}_m=\infty$ (and by convention, $T^{(W,r)}_m=\infty$). Thus, $\{\sigma^{(r)}_m<\infty\}$ is measurable with respect to $T^{(B,r)}_m$ and $T^{(W,r)}_m$. For every $k\in \N\cup\{\infty\}$ and every $p\in(0,1)$, we denote by $\mathsf{NB}(k,p)$ the negative binomial distribution with parameters $k$ and $p$ (where $\mathsf{NB}(\infty,p)$ is a Dirac mass at infinity).

\begin{Lem}\label{lem:LawBWUnderP}Let $p\in[p_c,1)$ and $n,m\in\N$ such that $n\geq m$. The following hold under $\Pb_p$.

\begin{itemize}
	\item \textbf{Left peeling.} $\{B^{(l)}_k : 0\leq k \leq T^{(B,l)}_n\}$ is a random walk with step distribution $\mu_p$, killed at the first entrance in $[n,+\infty)$. Conditionally on $T^{(B,l)}_n$, $T^{(W,l)}_n$ has distribution $\mathsf{NB}(T^{(B,l)}_n,g(p))$. Then, conditionally on $T^{(W,l)}_n$, $\{W^{(l)}_k : 0\leq k \leq T^{(W,l)}_n\}$ is a random walk with step distribution $\mu_{1-p}$, independent of $B^{(l)}$. Finally, conditionally on $T^{(B,l)}_n$ and $T^{(W,l)}_n$, $\{z^{(l)}_k : 1\leq k \leq \sigma^{(l)}_n-1\}$ is uniformly distributed among the set of binary sequences with $T^{(W,l)}_n$ zeros and $T^{(B,l)}_n-1$ ones.

	\item \textbf{Right peeling.} The above results hold when replacing $l$ by $r$, except that $\{B^{(r)}_k : 0\leq k \leq T^{(B,r)}_m\}$ is killed at the first entrance in $(-\infty,-m]$. Moreover, $\sigma^{(r)}_m$ is possibly infinite: conditionally on $T^{(B,r)}_m$ and $T^{(W,r)}_m$, $\{z^{(r)}_k : 1\leq k \leq \sigma^{(r)}_m-1\}$ is uniformly distributed among the set of binary sequences with $T^{(W,r)}_m$ zeros and $T^{(B,r)}_m-1$ ones on the event $\{\sigma^{(r)}_m<\infty\}$, and distributed as a sequence of i.i.d.\ variables with Bernoulli distribution of parameter $g(p)$ on the event $\{\sigma^{(r)}_m=\infty\}$.
\end{itemize} Lastly, $X^{(r)}$ is independent of $\{X^{(l)}_k : 0 \leq k \leq \sigma^{(l)}_n\}$. (However, it is not independent of the whole process $X^{(l)}$.)

\end{Lem}

\begin{proof}By the spatial Markov property, $X^{(r)}$ is independent the map revealed by the left peeling process, i.e.\ of $\{X^{(l)}_k : 0 \leq k \leq \sigma^{(l)}_n\}$. Moreover, $X^{(l)}$ and $X^{(r)}$ are distributed as the process $X$ of Definition \ref{def:ExplorationProcess}. We restrict our attention to the left peeling. By Lemma \ref{lem:LawBW}, $B^{(l)}$ and $W^{(l)}$ are independent random walks with step distribution $\mu_p$ and $\mu_{1-p}$, while $z^{(l)}$ is an independent sequence of i.i.d.\ variables with Bernoulli distribution of parameter $g(p)$. The first assertion follows from the definition of $T^{(B,l)}_n$. The random time $\sigma^{(l)}_n$ being measurable with respect to $X^{(l,1)}$, it is independent of $W^{(l)}$. Thus, conditionally on $T^{(B,l)}_n$ the lifetime of $W^{(l)}$ is exactly the number of failures before the $T^{(B,l)}_n$-th success in a sequence of independent Bernoulli trials with parameter $g(p)$. This is the second assertion. By definition, $Y^{(1,l)}_{\sigma^{(l)}_n} \neq 0$ so that $z^{(l)}_{\sigma^{(l)}_n}=1$. Conditionally on $T^{(B,l)}_n$ and $T^{(W,l)}_n$, the binary sequence $\{z^{(l)}_k : 1\leq k \leq \sigma^{(l)}_n-1\}$ is uniform among all possibilities by definition of the negative binomial distribution. The properties of the right peeling follow from a slight adaptation of these arguments.\end{proof}

We now describe the events we condition on to force an infinite critical open percolation cluster and define the $\IIC$. As in \cite{kesten_incipient_1986}, we use two distinct conditionings. Firstly, $B^{(l)}$ describes the length of the open segment on the boundary of the unexplored map (before the left peeling stops). This segment represents the revealed part of the left boundary of $\C$ (the open cluster of the origin). We define the height $h(\C)$ of $\C$ by
\begin{equation}\label{eqn:HeightCluster}
  h(\C):= \sup \left \lbrace B^{(l)}_k : 0\leq k\leq T \right \rbrace, \quad \text{where} \quad T:=\inf \left \lbrace k\in\Z_+ : B^{(l)}_k<0 \right \rbrace.
\end{equation} This is also the length of the loop-erasure of the open path revealed by the left peeling. Note that both the perimeter of the hull $\Hc$ of $\C$ and its size $\vert \C \vert$ are larger than $h(\C)$. Consistently, we want to condition the height of the cluster to be larger than $n$. In terms of the exploration process, this exactly means that $B^{(l)}$ reaches $[n,+\infty)$ before $(-\infty,0)$: 
\[\{h(\C) \geq n\}= \left \lbrace T^{(B,l)}_n<T \right \rbrace.\] We thus work under $\Pb(\cdot \mid h(\C) \geq n )$ and let $n$ go to infinity. Secondly, we let $p>p_c=1/2$ and condition $\C$ to be infinite (which is an event of positive probability under $\Pb_p$). In other words, we work under $\Pb_p(\cdot \mid \vert \C \vert = \infty )$ and let $p$ decrease towards $p_c$.

We now describe the law of the exploration process under these conditional probability measures, starting with $\Pb(\cdot \mid h(\C) \geq n )$. In the next part, we voluntarily choose the same parameter $n$ for the definition of the left peeling process and the conditioning $\{h(\C) \geq n\}$. This will make our argument simpler. The event $\left \lbrace h(\C) \geq n \right \rbrace=\{T^{(B,l)}_n<T\}$ is measurable with respect to $\{B^{(l)}_k : 0\leq k \leq T^{(B,l)}_n\}$. By Lemma \ref{lem:LawBWUnderP}, it is then independent of the other variables involved in the stopped left peeling process.

\begin{Lem}\label{lem:LawBWUnderPn}Let $n,m\in\N$ such that $n\geq m$. Under $\Pb(\cdot \mid h(\C) \geq n)$, the statements of Lemma \ref{lem:LawBWUnderP} (under $\Pb$) hold except that $\{B^{(l)}_k : 0\leq k \leq T^{(B,l)}_n\}$ is a random walk conditioned to reach $[n,+\infty)$ before $(-\infty,0)$ (and killed at that entrance time).
\end{Lem}

We now describe the exploration process under $\Pb_p(\cdot \mid \vert \C \vert = \infty )$, for $p>p_c$.

\begin{Lem}\label{lem:LawBWUnderPp}Let $p>p_c$ and $n,m\in\N$ such that $n\geq m$. Under $\Pb_p(\cdot \mid \vert \C \vert =\infty)$, the statements of Lemma \ref{lem:LawBWUnderP} (under $\Pb_p$) hold except that $\{B^{(l)}_k : 0\leq k \leq T^{(B,l)}_n\}$ is a random walk conditioned to stay nonnegative killed at the first entrance in $[n,+\infty)$, and $\{B^{(r)}_k : 0\leq k \leq T^{(B,r)}_m\}$ is a random walk conditioned to stay larger than $-n$ killed at the first entrance in $(-\infty,-m]$.
\end{Lem}

\begin{proof}We consider the left peeling first. From the proof of $p_c=1/2$ in~\cite{angel_scaling_2004}, we get that \[\{\vert \C \vert =\infty\}=\{T=\infty\} \quad \Pb_p\text{-a.s.}.\] Indeed, with probability one, if $T$ is finite the open segment on the boundary is swallowed by the exploration and $\C$ is confined in a finite region, while if $T$ is infinite the exploration reveals infinitely many open vertices connected to the origin by an open path. In particular, the event $\{\vert \C \vert =\infty\}$ is measurable with respect to $B^{(l)}$, and thus independent of $W^{(l)}$ and $z^{(l)}$ by Lemma \ref{lem:LawBW}. The first assertion follows.

We now focus on the right peeling process and denote by $M_\infty$ the underlying $\UIHPT$. The event $\{T^{(B,l)}_n<\infty\}\cap\{T^{(B,l)}_n<T\}$ has probability one under $\Pb_p(\cdot \mid \vert \C \vert = \infty )$. Thus, the right exploration is performed in a half-planar triangulation $M'_\infty$ with distribution $\Pb_p$ and a ``White-Black-White" boundary condition (the open segment on the boundary has size $n+1$ and its rightmost vertex is the origin of $M_\infty$). Since the stopped left peeling reveals a.s.\ finitely many vertices, the open cluster of the origin $\C$ in $M_\infty$ is infinite iff the open cluster of the origin $\C'$ in $M'_\infty$ is infinite. In other words, the right exploration is distributed as the exploration process of the $\UIHPT$ with the above boundary condition under $\Pb_p(\cdot \mid \vert \C \vert = \infty )$, and is independent of $\{X^{(l)}_k : 0 \leq k \leq \sigma^{(l)}_n\}$. By the same argument as above, \[\{\vert \C \vert =\infty\}=\{T'_{-n}=\infty\}\quad \Pb_p\text{-a.s.,}\quad  \text{where } T'_{-n}:=\inf \left \lbrace k\in\Z_+ : B^{(r)}_k<-n \right \rbrace.\] Again, $\{\vert \C \vert =\infty\}$ is measurable with respect to $B^{(r)}$, and thus independent of $W^{(r)}$ and $z^{(r)}$ by Lemma \ref{lem:LawBW}. This ends the proof.\end{proof}

\subsection{Distribution of the revealed map}\label{sec:RevealedMap}

We consider the peeling process of Algorithm \ref{alg:AlgorithmIIC} under the probability measures $\Pb(\cdot \mid h(\C) \geq n)$ and $\Pb_p(\cdot \mid \vert \C \vert = \infty )$. We denote by $M_\infty$ the underlying infinite triangulation of the half-plane, and by $\Hc$, $\Hc_l$ and $\Hc_r$ the percolation hulls of the origin and its left and right neighbours on the boundary of $M_\infty$.

We denote by $M_{n,m}$ the planar map that is revealed by the peeling process of Algorithm \ref{alg:AlgorithmIIC} with parameters $n,m\in\N$. Namely, the vertices and edges of $M_{n,m}$ are those discovered by the stopped peeling process (including the swallowed regions). However, by convention, we exclude the edges and vertices discovered at the last step of both the left and right (stopped) peeling processes. This will make our description simpler. The root edge of $M_{n,m}$ is the root edge of $M_\infty$. By definition, $M_{n,m}$ is infinite on the event $\{\sigma^{(r)}_m=\infty\}$. The goal of this section is to provide a decomposition of the map $M_{n,m}$.

\bigskip

\noindent\textbf{The percolation hulls.} The origin of $M_\infty$ and its left and right neighbours on the boundary belong to $M_{n,m}$. We keep the notation $\Hc$, $\Hc_l$ and $\Hc_r$ for the associated percolation hulls in $M_{n,m}$. We let \[\underline{W}^{(l)}:=\min_{0\leq k \leq T^{(W,l)}_n}W^{(l)}_k \quad \text{and} \quad  \underline{W}^{(r)}:=\min_{0\leq k \leq T^{(W,r)}_m}W^{(r)}_k,\] and define the finite sequences $w^{(l)}$ and $w^{(r)}$ by 
\begin{equation}\label{eqn:ProcessWmin}
w^{(l)}_k=\left\lbrace
\begin{array}{ccc}
k & \mbox{if} & -\underline{W}^{(l)}\leq k < 0\\
W^{(l)}_{k} & \mbox{if} & 0\leq k\leq T^{(W,l)}_n\\
\end{array}\right. \quad \text{and} \quad w^{(r)}_k=\left\lbrace
\begin{array}{ccc}
k & \mbox{if} & -\underline{W}^{(r)}\leq k<0\\
W^{(r)}_{-k} & \mbox{if} & 0\leq k\leq T^{(W,r)}_m\\
\end{array}\right..
\end{equation} We define equivalence relations on $\{-T^{(W,l)}_n, \ldots, -\underline{W}^{(l)}\}$ and $\{-T^{(W,r)}_m,\ldots,-\underline{W}^{(r)}\}$ by applying \eqref{eqn:EquivC} with $w^{(l)}$ and $w^{(r)}$. By the construction of Section \ref{sec:ContourFunction}, this defines two planar maps which are not looptrees in general. Nonetheless, we denote them by $\Lt_{w^{(l)}}$ and $\Lt_{w^{(r)}}$ to keep the notation simple. We replace $\Lt_{w^{(l)}}$ by its image under a reflection. The left (resp.\ right) boundary of $\Lt_{w^{(l)}}$ is the projection of nonpositive (resp.\ nonnegative) integers of $\{-T^{(W,l)}_n, \ldots, -\underline{W}^{(l)}\}$ on $\Lt_{w^{(l)}}$. The same holds when replacing $l$ by $r$, up to inverting left and right boundaries. We choose the root edges consistently. The sequence $w^{(l)}$  and $\Lt_{w^{(l)}}$ are a.s.\ finite, while $w^{(r)}$ and thus $\Lt_{w^{(r)}}$ are infinite on the event $\{\sigma^{(r)}_m=\infty\}$.
We introduce the finite sequences $b^{(l)}$ and $b^{(r)}$ defined by
\begin{equation}\label{eqn:ProcessBmin}
	b^{(l)}_k=B^{(l)}_k, \quad 0\leq k < T^{(B,l)}_n \quad \text{and} \quad b^{(r)}_k=B^{(r)}_k, \quad 0\leq k < T^{(B,r)}_m.
\end{equation} We define an equivalence relation on $\{-T^{(B,l)}_n+1,\ldots,T^{(B,r)}_m-1  \}$ by applying (\ref{eqn:EquivCStar}), with $b^{(l)}$ (resp.\ $b^{(r)}$) playing the role of $C'$ (resp.\ $C$). By the construction of Section \ref{sec:ContourFunction}, this defines a planar map that we denote by $\Lt_{b^{(l)},b^{(r)}}$, which is not a looptree either. (It is also obtained by defining $\Lt_{b^{(r)}}$ and the finite forest of looptrees $\mathbf{F}_{b^{(l)}}$ as above, and gluing them along their boundaries as in Section \ref{sec:ContourFunction}.) The left (resp.\ right) boundary of $\Lt_{b^{(l)},b^{(r)}}$ is the projection of nonpositive (resp.\ nonnegative) integers of $\{-T^{(B,l)}_n+1,\ldots,T^{(B,r)}_m-1  \}$ (and we choose the root edge consistently). Again, $\Lt_{b^{(l)},b^{(r)}}$ is infinite on the event $\{\sigma^{(r)}_m=\infty\}$.

\begin{Prop}\label{prp:HullRevealedMap} Let $n,m\in\N$ such that $n\geq m$. Under $\Pb(\cdot \mid h(\C) \geq n)$ and $\Pb_p(\cdot \mid \vert \C \vert = \infty )$, in the map $M_{n,m}$, the percolation hulls $\Hc$, $\Hc_l$ and $\Hc_r$ are the independent random maps $\Lt_{b^{(l)},b^{(r)}}$, $\Lt_{w^{(l)}}$ and $\Lt_{w^{(r)}}$ in which each internal face of degree $l\geq 2$ is filled in with an independent triangulation with distribution $\mathbb{W}_l$ equipped with a Bernoulli percolation model with parameter $p_c$ (resp.\ $p$).\end{Prop}

\begin{proof}The proof follows the same lines as Proposition \ref{prp:BoundaryPercolationHull}, to which we refer for more details. We begin with the left peeling, which follows the percolation interface between $\Hc_l$ and $\Hc$. On the one hand, the right contour of $\Hc_l$ is encoded by $\{W^{(l)}_k : 0\leq k \leq T^{(W,l)}_n\}$ and $-\underline{W}^{(l)}$ vertices are discovered on the left boundary of $M_\infty$. We obtain the map $\Lt_{w^{(l)}}$. On the other hand, the left contour of $\Hc$ is encoded by $b^{(l)}$, which defines the forest of finite looptrees $\mathbf{F}_{b^{(l)}}$ (by applying \eqref{eqn:EquivC}). We now deal with the right peeling, that starts a.s.\ in a triangulation of the half-plane with an open segment of size $n+1$ on the boundary. This segment corresponds (up to the last vertex) to the right boundary of $\mathbf{F}_{b^{(l)}}$, i.e.\ to the set \begin{equation}\label{eqn:OpenSegment}
	\mathsf{V}_l:=\left\lbrace 0\leq k<T^{(B,l)}_n : B^{(l)}_{k}=\inf_{k\leq i< T^{(B,l)}_n}{B^{(l)}_i} \right \rbrace.
\end{equation} The right peeling follows the percolation interface between $\Hc$ and $\Hc_r$. In particular, the right contour of $\Hc$ is encoded by $b^{(r)}$. By construction, vertices associated to the set
\begin{equation}
\label{eqn:OpenSegmentR}
	\mathsf{V}_r:=\left\lbrace 0\leq k<T^{(B,r)}_m : B^{(r)}_{k}=\inf_{0\leq i\leq k}{B^{(r)}_i} \right \rbrace
\end{equation} are identified with the right boundary of $\mathbf{F}_{b^{(l)}}$. Precisely, every $k\in \mathsf{V}_r$ such that $B^{(r)}_{k}=-j$ is matched to the $(j+1)$-th element of $\mathsf{V}_l$ (note that $-\inf(b^{(r)})\leq m\leq n$). We obtain the map $\Lt_{b^{(l)},b^{(r)}}$. Finally, $\{W^{(r)}_k : 0\leq k \leq T^{(W,r)}_m\}$ encodes the left contour of $\Hc_r$ and $-\underline{W}^{(r)}$ vertices are discovered on the right boundary of $M_\infty$, which gives the map $\Lt_{w^{(r)}}$. 

The spatial Markov property shows that the finite faces of $\Lt_{b^{(l)},b^{(r)}}$, $\Lt_{w^{(l)}}$ and $\Lt_{w^{(r)}}$ are filled in with independent percolated Boltzmann triangulations with a simple boundary (the boundary conditions are fixed by the hulls and the percolation parameter by the underlying model). Since $M_{n,m}$ is the map revealed by the peeling process, we then recover the whole percolation hulls $\Hc$, $\Hc_l$ and $\Hc_r$.\end{proof}

We now focus on the connection between the percolation hulls in $M_{n,m}$. In order to make the next statement simpler, we generalize the definition of the uniform necklace.

\begin{Def}\label{def:UniformNecklace}
	Let $x,y\in \N$ and $\{z_i : 1\leq i \leq x+y\}$ uniform among the set of binary sequence with $x$ ones and $y$ zeros. Define	
	\[S_k:=\sum_{i=1}^{k}{z_i}, \quad 1\leq k \leq x+y.\] The uniform necklace of size $(x,y)$ is obtained from the graph of $[-x,y+1]\cap\Z$ by adding the set of edges $\left\lbrace \left( -S_k,k+1-S_k \right) : 1\leq k \leq x+y \right\rbrace$. Its distribution is denoted by $\mathsf{UN}(x,y)$.\end{Def} By convention, for $x=y=\infty$, $\mathsf{UN}(\infty,\infty)$ is the uniform infinite necklace of Section \ref{sec:StatementResults}.

We now use a construction similar to that preceding Proposition \ref{prp:UniformNecklace}. Let $V_l:=\{c_i : 1\leq i \leq T^{(W,l)}_n+1 \}$ be the corners of the right boundary of $\Hl$ listed in contour order, and similarly for $V_r:=\{c'_i : 1\leq i \leq T^{(B,l)}_n\}$ with the left boundary of $\Hc$. Then, let $N_l$ be the planar map with vertex set $V_l \cup V_r$, such that two vertices are neighbours iff the associated corners are connected by an edge in the $\UIHPT$. The planar map $N_r$ is defined symmetrically with the right boundary of $\Hc$ and the left boundary of $\Hr$.

\begin{Prop}\label{prp:UniformNecklaceIIC}
Let $n,m\in\N$ such that $n\geq m$. Under $\Pb(\cdot \mid h(\C) \geq n)$ and $\Pb_p(\cdot \mid \vert \C \vert = \infty )$, conditionally on $T^{(B,l)}_n$, $T^{(B,r)}_m$, $T^{(W,l)}_n$ and $T^{(W,r)}_m$, the following holds: $N_l$ and $N_r$ are independent uniform necklaces with distribution $\mathsf{UN}(T^{(W,l)}_n, T^{(B,l)}_n-1)$ and $\mathsf{UN}(T^{(B,r)}_m-1, T^{(W,r)}_m)$. Otherwise said, in the map $M_{n,m}$, $\Hc_l$, $\Hc$ and $\Hc_r$ are glued along a pair of independent uniform necklaces with respective size $(T^{(W,l)}_n, T^{(B,l)}_n-1)$ and $(T^{(B,r)}_m-1, T^{(W,r)}_m)$.
\end{Prop} 

\begin{proof}We follow the arguments of Proposition \ref{prp:UniformNecklace}. For every $k\in\Z_+$ such that $Y^{(l,1)}_k\neq 0$, there is an edge between the revealed open corner of the left boundary of $\Hc$ and the last revealed closed corner of the right boundary of $\Hc_l$. The converse occurs when $Y^{(l,2)}_k\neq 0$. Then, $N_l$ is the uniform necklace generated by $\{z^{(l)}_k : 1\leq k < \sigma^{(l)}_n\}$. Similarly, $N_r$ is the uniform necklace generated by $\{z^{(r)}_k : 1\leq k < \sigma^{(r)}_m\}$. We conclude by Lemmas \ref{lem:LawBWUnderPn} and \ref{lem:LawBWUnderPp}.\end{proof}

We obtain a decomposition of $M_{n,m}$ illustrated in Figure \ref{fig:DecompositionFiniteIIC} (in the finite case). The map $M_{n,m}$ is measurable with respect to the processes $w^{(l)}$, $w^{(r)}$, $b^{(l)}$ and $b^{(r)}$, the variables $z^{(l)}$ and $z^{(r)}$ defining the uniform necklaces and the percolated Boltzmann triangulations with a simple boundary that fill in the internal faces.

\begin{figure}[!ht]
\centering
\includegraphics[scale=1.7]{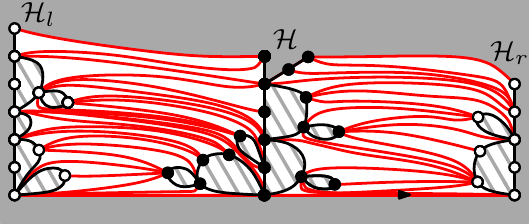}
\caption{The decomposition of the map $M_{n,m}$ into percolation hulls and necklaces.}
\label{fig:DecompositionFiniteIIC}
\end{figure}

\subsection{The $\normalfont{\IIC}$ probability measure}\label{sec:IICMeasure}

In this section, we define an infinite triangulation of the half-plane with distribution $\Piic$. We use infinite looptrees, uniform necklaces and Boltzmann triangulations as building blocks.

\bigskip

\noindent\textbf{Definition of the $\normalfont{\IIC}$.} Let $\Psf$ be a probability measure and under $\Psf$, let $A^{(l)}:=\{A^{(l)}_k : k\in\Z_+\}$ be a random walk with law $P^{\mu \uparrow}_0$. Let also $A^{(r)}:=\{A^{(r)}_k : k\in\Z_+\}$, $V^{(l)}:=\{V^{(l)}_k : k\in\Z_+\}$ and $V^{(r)}:=\{V^{(r)}_k : k\in\Z_+\}$ be random walks with law $P^{\mu}_0$. Finally, let $y^{(l)}=\{y^{(l)}_k : k\in\N\}$ and $y^{(r)}=\{y^{(r)}_k : k\in\N\}$ be sequences of i.i.d.\ variables with Bernoulli distribution of parameter $1/2$. We assume that these processes are all independent under $\Psf$.

Following Section \ref{sec:ContourFunctionRandom}, we define the looptrees $\Lt_{A^{(l)},A^{(r)}}$, $\Lt_{V^{(l)}}$ and $\Lt_{V^{(r)}}$. We replace $\Lt_{V^{(r)}}$ by its image under a reflection (with the root edge going from the vertex $0$ to $-1$). By Propositions \ref{prp:ComponentTree} and \ref{prp:ComponentTreeBiInfinite}, the associated trees of components $\Tr_{A^{(l)},A^{(r)}}$, $\Tr_{V^{(l)}}$ and $\Tr_{V^{(r)}}$ have respective distribution $\GWmm^\infty$, $\GWmm^{(\infty,l)}$ and $\GWmm^{(\infty,r)}$. We agree that the vertices of $\Lt_{A^{(l)},A^{(r)}}$ are open, while vertices of $\Lt_{V^{(l)}}$ and $\Lt_{V^{(r)}}$ are closed. For every vertex $v\in (\Tr_{A^{(l)},A^{(r)}})_\bullet$ (resp.\ $(\Tr_{V^{(l)}})_\bullet$, $(\Tr_{V^{(r)}})_\bullet$) we let $M_v$ (resp.\ $M^l_v$, $M^r_v$) be a Boltzmann triangulation with distribution $\mathbf{W}_{\deg(u)}$ (independent of all the other variables). Then, $\Hc$, $\Hl$ and $\Hr$ are defined by 
\[\Hc:=\Phi^{-1}\left(\Tr_{A^{(l)},A^{(r)}}, \left\lbrace M_v : v \in \left(\Tr_{A^{(l)},A^{(r)}}\right)_\bullet\right\rbrace\right) , \quad  \Hl:=\Phi^{-1}\left(\Tr_{V^{(l)}}, \left\lbrace M^l_v : v \in \left(\Tr_{V^{(l)}}\right)_\bullet\right\rbrace\right),\] and similarly for $\Hr$ replacing $l$ by $r$. Finally, we let $\mathbf{N}_l$ and $\mathbf{N}_r$ be the uniform infinite necklaces with distribution $\mathsf{UN}(\infty,\infty)$ generated by $y^{(l)}$ and $y^{(r)}$. The infinite planar map $M_\infty$ is defined by gluing $(\Hl,\Hc,\Hr)$ along $(\mathbf{N}_l,\mathbf{N}_r)$, i.e.\
\[M_\infty:= \Psi_{(\mathbf{N}_l,\mathbf{N}_r)}(\Hl,\Hc,\Hr).\] In particular, the root edge of $M_\infty$ connects the origin of $\Lt_{A^{(l)},A^{(r)}}$ to that of $\Lt_{V^{(r)}}$. The probability measure $\Piic$ is the distribution of $M_\infty$ under $\Psf$. 

By construction, $M_\infty$ has a proper embedding in the plane and an infinite boundary isomorphic to $\Z$. The infinite looptrees having a unique spine, $M_\infty$ is also one-ended a.s.. More precisely, $M_\infty$ is a.s.\ a triangulation of the half-plane, i.e.\ the proper embedding of an infinite and locally finite connected graph in the upper half-plane $\mathbb{H}$ whose faces are finite and have degree three (with no self-loop). By definition, vertices of $M_\infty$ are coloured and $M_\infty$ has the ``White-Black-White" boundary condition of Figure \ref{fig:InitialColouringIIC}. Moreover, the planar maps $\Hc$, $\Hc_l$ and $\Hc_r$ are the percolation hulls of the origin vertex and its neighbours on the boundary, which justifies the choice of notation.

\bigskip

\noindent\textbf{Exploration of the $\normalfont{\IIC}$.} For every $n,m\in \N$, we define a finite map $M_{n,m}$ under $\Psf$, by replicating the construction of Section \ref{sec:RevealedMap}. Let 
\begin{equation}\label{eqn:TTimeA}
  T^{(A,l)}_n:=\inf \left\lbrace k\in\Z_+ : A^{(l)}_k \geq n \right\rbrace \quad \text{and} \quad T^{(A,r)}_m:=\inf \left\lbrace k\in\Z_+ : A^{(r)}_k \leq -m \right\rbrace,
\end{equation} and conditionally on $T^{(A,l)}_n$ and $T^{(A,r)}_m$, let $T^{(V,l)}_n$ and $T^{(V,r)}_m$ be independent random variables with distribution $\mathsf{NB}(T^{(A,l)}_n,1/2)$ and $\mathsf{NB}(T^{(A,r)}_m,1/2)$. We let 
\[\rho^{(l)}_n:=T^{(A,l)}_n+T^{(V,l)}_n  \quad \text{and} \quad \rho^{(r)}_m:=T^{(A,r)}_m+T^{(V,r)}_m.\] We define $v^{(l)}$ and $v^{(r)}$ as in \eqref{eqn:ProcessWmin} (replacing $W$ by $V$), and $a^{(l)}$ and $a^{(r)}$ as in \eqref{eqn:ProcessBmin} (replacing $B$ by $A$). Note that all the random times considered here are finite $\Psf$-a.s.. Let us consider the finite planar maps $\Lt_{a^{(l)},a^{(r)}}$, $\Lt_{v^{(l)}}$ and $\Lt_{v^{(r)}}$, defined according to the construction of Section \ref{sec:RevealedMap}. As we will see in Proposition \ref{prp:Submap}, these are possibly (though not always) sub-maps of the infinite looptrees $\Lt_{A^{(l)},A^{(r)}}$, $\Lt_{V^{(l)}}$ and $\Lt_{V^{(r)}}$. We now fill in each internal face of degree $l\geq 2$ of the finite maps with an independent percolated Boltzmann triangulation with distribution $\mathbf{W}_l$. We agree that we use the same triangulations to fill in corresponding faces in the finite maps and their infinite counterparts when this is possible. Finally, we glue the right boundary of $\Lt_{v^{(l)}}$ and the left boundary of $\Lt_{a^{(l)},a^{(r)}}$ along the uniform necklace with size $(T^{(V,l)}_n, T^{(A,l)}_n-1)$ generated by $\{y^{(l)}_k : 1\leq k < \rho^{(l)}_n\}$. Similarly, we glue the left boundary of $\Lt_{v^{(r)}}$ and the right boundary of $\Lt_{a^{(l)},a^{(r)}}$ along the uniform necklace with size $(T^{(V,r)}_m, T^{(A,r)}_m-1)$ generated by $\{y^{(r)}_k : 1\leq k < \rho^{(r)}_m\}$. These gluing operations are defined as in Section \ref{sec:StatementResults} provided minor adaptations to the finite setting. The resulting planar map is denoted by $M_{n,m}$. 

\begin{Prop}\label{prp:Submap} A.s., for every $n\in\N$ and every $m\in \N$ such that $m\leq \inf\{ A^{(l)}_k : k \geq T^{(A,l)}_n\}$, $M_{n,m}$ is a sub-map of $M_\infty$.\end{Prop}

\begin{proof}Let $n,m\in \N$. We first consider the equivalence relation $\sim$ defined by applying (\ref{eqn:EquivC}) to $V^{(l)}$. The relation $i\sim j$ is determined by the values $\{V^{(l)}_k : i\wedge j\leq k \leq i\vee j\}$. Thus, the restriction of $\sim$ to $[-T^{(V,l)}_n,-\underline{V}^{(l)}]\cap \Z$ is isomorphic to $\Lt_{v^{(l)}}$, and $\Lt_{v^{(l)}}$ is a sub-map of $\Lt_{V^{(l)}}$. The same argument proves that $\Lt_{v^{(r)}}$ is a sub-map of $\Lt_{V^{(r)}}$, and that $\mathbf{F}_{a^{(l)}}$ (resp.\ $\Lt_{a^{(r)}}$) is a sub-map of the forest $\mathbf{F}_{A^{(l)}}$ (resp.\ the looptree $\Lt_{A^{(r)}}$). Moreover, the finite necklaces generated by $\{y^{(l)}_k : 1\leq k < \rho^{(l)}_n\}$ and $\{y^{(r)}_k : 1\leq k < \rho^{(r)}_m\}$ are sub-maps of the infinite uniform necklaces $\mathbf{N}_l$ and $\mathbf{N}_r$ generated by $y^{(l)}$ and $y^{(r)}$ in $M_\infty$. 

The only thing that remains to check is the gluing of $\mathbf{F}_{a^{(l)}}$ and $\Lt_{a^{(r)}}$ into $\Lt_{a^{(l)},a^{(r)}}$. Let $a:=\{a_k :-T^{(A,l)}_n+1 \leq k \leq T^{(A,r)}_m-1 \}$ and $A:=\{A_k : k\in \Z\}$ be defined from $(a^{(l)},a^{(r)})$ and $(A^{(l)},A^{(r)})$ as in (\ref{eqn:ProcessCStar}). We let $\sim_a$ and $\sim_A$ be the associated equivalence relations, defined from (\ref{eqn:EquivCStar}). When $m$ is too large compared to $n$, $\sim_a$ is finer than the restriction of $\sim_A$ to $[ -T^{(A,l)}_n+1,T^{(A,r)}_m-1 ]\cap\Z$ (see Figure \ref{fig:GluingLeftRightBlackHull} for an example). In other words, $\Lt_{a^{(l)},a^{(r)}}$ cannot be realized as a subset of $\Lt_{A^{(l)},A^{(r)}}$. By definition of $\sim_A$, such a situation occurs only if 
\begin{equation}\label{eqn:ConditionNonIntersection}
  -\inf_{0 \leq k \leq T^{(A,r)}_m-1 }{A^{(r)}_k}> \inf_{k \geq T^{(A,l)}_n-1}{A^{(l)}_k},
\end{equation} which is avoided for $m\leq \inf\{ A^{(l)}_k : k \geq T^{(A,l)}_n\}$. This concludes the proof.\end{proof}

\begin{figure}[!ht]
\centering
\includegraphics[scale=1.2]{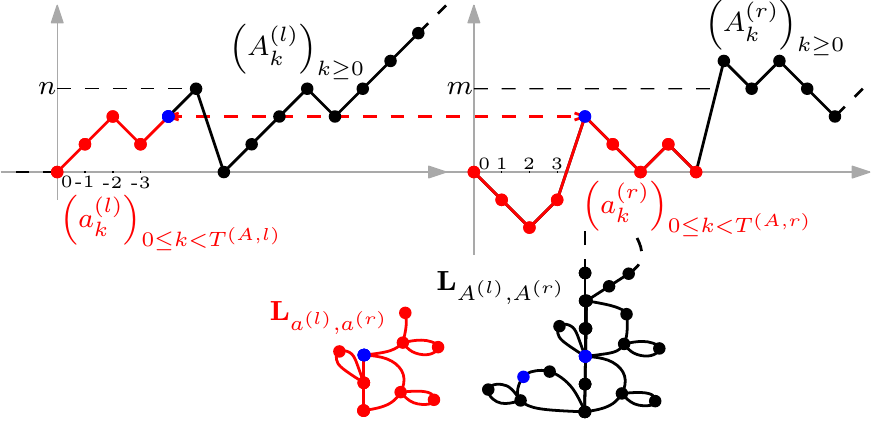}
\caption{A configuration where $M_{n,m}$ cannot be realized as a subset of $M_\infty$.}
\label{fig:GluingLeftRightBlackHull}
\end{figure}

\bigskip

\noindent\textbf{Convergence of the exploration processes.} The goal of this paragraph is to prove that the distributions of the maps $M_{n,m}$ are close in total variation distance under $\mathsf{P}$, $\Pb(\cdot \mid h(\C) \geq n)$ and $\Pb_p(\cdot \mid \vert \C \vert = \infty )$ for a suitable choice of $n$, $m$, and $p$. In the next part, $\mathcal{F}$ denotes the set of Borel functions from $\Z^\N$ to $[0,1]$. For every process $Z$ and every stopping time $T$, $Z':=\{Z_k : 0\leq k\leq T\}$ is interpreted as an element of $\Z^\N$ by putting $Z'_k=0$ for $k>T$. We use the notation of Section \ref{sec:RandomWalk} and start with two preliminary lemmas.

\begin{Lem}\label{lem:DistanceTVLeftExploration}For every $n,m\in\N$ such that $n\geq m$, 
\[\lim_{p\downarrow p_c}\sup_{F\in \mathcal{F}} \left| \mathbb{E}_p\left( F\left( B_k^{(l)} : 0\leq k \leq T^{(B,l)}_n \right) \ \middle| \ \vert \C \vert=\infty \right)-E^{\mu \uparrow}_0\left( F\left( Z_k : 0\leq k \leq T_{n} \right) \right) \right| =0. \] 
\end{Lem}

\begin{proof}Let $\varepsilon>0$, $n,m\in\N$ such that $n\geq m$ and $F\in \mathcal{F}$. By Lemma \ref{lem:LawBWUnderPp},
\[\mathbb{E}_p\left( F\left( B_k^{(l)} : 0\leq k \leq T^{(B,l)}_n \right) \ \middle| \ \vert \C \vert=\infty\right)=E_0^{\mu_p \uparrow}\left( F\left( Z_k : 0\leq k \leq T_{n} \right) \right).\] Furthermore, by Lemma \ref{lem:ConvergenceRWCSPDrift}, for every $K\in\N$, 
\[P_0^{\mu_p \uparrow}(T_{n}> K)\underset{p\downarrow p_c}{\longrightarrow}P_0^{\mu \uparrow}(T_n> K).\] Since $T_n$ is finite $P_0^{\mu \uparrow}$-a.s., up to choosing $K$ large enough and then $p$ close enough to $p_c$,
\[\max \left(P_0^{\mu_p \uparrow}(T_{n}> K),P_0^{\mu \uparrow}(T_{n}> K)\right)\leq \varepsilon.\]Then,
	\begin{align*}
		\sup_{F\in \mathcal{F}} \left| \mathbb{E}_p\left( F\left( B_k^{(l)} : 0\leq k \leq T^{(B,l)}_n \right) \ \middle| \ \vert \C \vert=\infty\right)-E^{\mu \uparrow}_0\left( F\left( Z_k : 0\leq k \leq T_{n} \right) \right) \right|\\
		\leq 2\varepsilon + \sup_{F:\Z^{K+1}\rightarrow[0,1]} \left| E_0^{\mu_p \uparrow}\left( F\left( Z_0, \ldots, Z_K \right) \right)-E^{\mu \uparrow}_0\left( F\left( Z_0,\ldots,Z_K \right) \right) \right|
	\end{align*} and by Lemma \ref{lem:ConvergenceRWCSPDrift},
\[\sup_{F:\Z^{K+1}\rightarrow[0,1]} \left| E_0^{\mu_p \uparrow}\left( F\left( Z_0,\ldots,Z_K \right) \right)-E^{\mu \uparrow}_0\left( F\left( Z_0,\ldots,Z_K \right) \right) \right|\underset{p\downarrow p_c}{\longrightarrow}0. \] 	This concludes the proof since $\varepsilon$ is arbitrary.\end{proof}

\begin{Lem}\label{lem:DistanceTVRightExploration}For every $\varepsilon>0$ and $m\in\N$, there exists $N\in\N$ such that for every $n\geq N$
\[\limsup_{p\downarrow p_c}\sup_{F\in \mathcal{F}} \left| \mathbb{E}_p\left( F\left( B_k^{(r)} : 0\leq k \leq T^{(B,r)}_m \right) \ \middle| \ \vert \C \vert=\infty\right)-E^\mu_0\left( F\left( Z_k : 0\leq k \leq T_{-m} \right) \right) \right| \leq \varepsilon. \]
\end{Lem}

\begin{proof}Let $\varepsilon>0$, $m\in\N$, $n\in \N$ such that $n\geq m$ and $F\in \mathcal{F}$. By Lemma \ref{lem:LawBWUnderPp} and translation invariance,
\begin{align*}\mathbb{E}_p\left( F\left( B_k^{(r)} : 0\leq k \leq T^{(B,r)}_m \right)\ \middle| \ \vert \C \vert=\infty \right)&=E_0^{\mu_p}\left( F\left( Z_k : 0\leq k \leq T_{-m} \right) \mid T_{-n}=\infty \right)\\
&=E_n^{\mu_p}\left( F\left( Z_k-n : 0\leq k \leq T_{n-m} \right) \mid T=\infty \right)\\
&=E_n^{\mu_p \uparrow}\left( F\left( Z_k-n : 0\leq k \leq T_{n-m} \right)\right).
\end{align*} Furthermore, by Lemma \ref{lem:ConvergenceRWCSPDrift} and then Lemma \ref{lem:ConvergenceRWCSPLargeX}, for every $K\in\N$, 
\[P_n^{\mu_p \uparrow}(T_{n-m}> K)\underset{p\downarrow p_c}{\longrightarrow}P_n^{\mu \uparrow}(T_{n-m}> K)\underset{n \rightarrow \infty}{\longrightarrow}P_0^{\mu}(T_{-m}> K).\] Since $T_{-m}$ is finite $P_0^{\mu}$-a.s., up to choosing $K$ large enough, then $n$ large enough and finally $p$ close enough to $p_c$,
\[\max \left(P_n^{\mu_p \uparrow}(T_{n-m}> K),P_0^{\mu}(T_{-m}> K)\right)\leq \varepsilon.\]Then,
	\begin{align*}
		\sup_{F\in \mathcal{F}} &\left| \mathbb{E}_p\left( F\left( B_k^{(r)} : 0\leq k \leq T^{(B,r)}_m \right) \ \middle| \ \vert \C \vert=\infty\right)-E^\mu_0\left( F\left( Z_k : 0\leq k \leq T_{-m} \right)\right) \right|\\
		\leq 2\varepsilon &+ \sup_{F:\Z^{K+1}\rightarrow[0,1]} \left| E_n^{\mu_p \uparrow}\left( F\left( Z_0-n, \ldots, Z_K-n \right) \right)-E^{\mu}_0\left( F\left( Z_0,\ldots,Z_K \right) \right) \right|\\
		\leq 2\varepsilon &+ \sup_{F:\Z^{K+1}\rightarrow[0,1]} \left| E_n^{\mu_p \uparrow}\left( F\left( Z_0-n, \ldots, Z_K-n \right) \right)-E^{\mu \uparrow}_n\left( F\left( Z_0-n,\ldots,Z_K-n \right) \right) \right|\\
		&+ \sup_{F:\Z^{K+1}\rightarrow[0,1]} \left| E^{\mu \uparrow}_n\left( F\left( Z_0-n,\ldots,Z_K-n \right) \right)-E^{\mu}_0\left( F\left( Z_0,\ldots,Z_K \right) \right) \right|
	\end{align*} By Lemma \ref{lem:ConvergenceRWCSPLargeX}, there exists $N\in\N$ such that for every $n\geq N$, 
	\[\sup_{F:\Z^{K+1}\rightarrow[0,1]} \left| E^{\mu \uparrow}_n\left( F\left( Z_0-n,\ldots,Z_K-n \right) \right)-E^{\mu}_0\left( F\left( Z_0,\ldots,Z_K \right) \right) \right|\leq \varepsilon,\] and by Lemma \ref{lem:ConvergenceRWCSPDrift}, for every $n\in\N$,
	\[\sup_{F:\Z^{K+1}\rightarrow[0,1]} \left| E_n^{\mu_p \uparrow}\left( F\left( Z_0-n, \ldots, Z_K-n \right) \right)-E^{\mu \uparrow}_n\left( F\left( Z_0-n,\ldots,Z_K-n \right) \right) \right| \underset{p\downarrow p_c}{\longrightarrow}0.\] This concludes the proof.\end{proof}
	
In what follows, $\mathcal{A}$ denotes the Borel $\sigma$-field of the local topology. 

\begin{Prop}\label{prp:Coupling}For every $\varepsilon>0$ and $m\in\N$, there exists $N\in\N$ such that for every $n\geq N$,  
\[ \limsup_{p\downarrow p_c}{\sup_{A\in \mathcal{A}}{\left\vert \Pb_p(M_{n,m}\in A \mid \vert \C \vert =\infty)-\Psf(M_{n,m}\in A) \right\vert}}\leq \varepsilon.\] Moreover, for every  $n,m\in\N$ such that $n\geq m$, the random planar map $M_{n,m}$ has the same law under $\Pb(\cdot \mid h(\C) \geq n)$ and $\Psf$.
\end{Prop}

\begin{proof}We start with the first assertion. Throughout this proof, we use $\Pb_p^\infty:=\Pb_p(\cdot \mid \vert \C \vert = \infty )$ in order to shorten the notation. Let $\varepsilon>0$ and $m\in\N$. By Lemma \ref{lem:DistanceTVRightExploration} and the definition of the $\IIC$, there exists $N\in\N$ such that for every $n\geq N$,
\begin{equation}\label{eqn:Coupling1}\limsup_{p\downarrow p_c}\sup_{F\in\mathcal{F}} \left| \mathbb{E}^\infty_p\left( F\left( B_k^{(r)} : 0\leq k \leq T^{(B,r)}_m  \right) \right)-\mathsf{E}\left( F\left( A_k^{(r)} : 0\leq k \leq T^{(A,r)}_m \right) \right) \right| \leq \varepsilon.\end{equation} We now fix $n\geq N$, and by Lemma \ref{lem:DistanceTVLeftExploration}
\begin{equation}\label{eqn:Coupling2}
	\lim_{p\downarrow p_c}\sup_{F\in \mathcal{F}} \left| \mathbb{E}^\infty_p\left( F\left( B_k^{(l)} : 0\leq k \leq T^{(B,l)}_n \right) \right)-\mathsf{E}\left( F\left( A_k^{(l)} : 0\leq k \leq T^{(A,l)}_n \right) \right) \right| =0.
\end{equation} Furthermore, by Lemma \ref{lem:LawBWUnderPp} and the definition of the $\IIC$, for every $K\in \N$ and $F\in\mathcal{F}$,
\begin{equation}\label{eqn:Coupling3}
\mathbb{E}^\infty_p\left( F\left( W_k^{(l)} : 0\leq k \leq T^{(W,l)}_n \right)\ \middle| \ T^{(B,l)}_n=K \right)=\mathsf{E}\left( F\left( V_k^{(l)} : 0\leq k \leq T^{(V,l)}_n \right) \ \middle| \ T^{(A,l)}_n=K \right).
\end{equation} For the same reason, for every $K,K'\in \Z_+$ and $F\in\mathcal{F}$, 
\begin{multline}\label{eqn:Coupling4}
\mathbb{E}^\infty_p\left( F\left( z_k^{(l)} : 0\leq k \leq \sigma^{(l)}_n \right) \ \middle| \ \left(T^{(B,l)}_n,T^{(W,l)}_n\right)=(K,K') \right)\\
=\mathsf{E}\left( F\left( y_k^{(l)} : 0\leq k \leq \rho^{(l)}_n \right)\ \middle| \ \left(T^{(A,l)}_n,T^{(V,l)}_n\right)=(K,K') \right).
\end{multline} The assertions \eqref{eqn:Coupling3} and \eqref{eqn:Coupling4} hold when replacing $l$ by $r$. Let $M'_{n,m}$ be the planar map obtained from the construction of $M_{n,m}$ in Sections \ref{sec:RevealedMap} and \ref{sec:IICMeasure} without filling the internal faces with Boltzmann triangulations with a simple boundary. Using \eqref{eqn:Coupling1}, \eqref{eqn:Coupling2}, \eqref{eqn:Coupling3} and \eqref{eqn:Coupling4}, we get that 
\begin{equation}\label{eqn:CouplingMPrime}
	\limsup_{p\downarrow p_c}{\sup_{A\in \mathcal{A}}{\left\vert \Pb_p^\infty(M'_{n,m}\in A)-\Psf(M'_{n,m}\in A) \right\vert}}\leq \varepsilon.
\end{equation} Since the colouring of the vertices in the Boltzmann triangulations filling in the faces of $M'_{n,m}$ is a Bernoulli percolation with parameter $p$ (resp. $p_c$), for every finite map $\m\in\M_f$ we have
\[ \limsup_{p\downarrow p_c}{\sup_{A\in \mathcal{A}}{\left\vert \Pb_p^\infty(M_{n,m}\in A \mid M'_{n,m}=\m)-\Psf(M_{n,m}\in A \mid M'_{n,m}=\m) \right\vert}}=0.\]
Since $M'_{n,m}$ is $\Psf$-a.s. finite, together with \eqref{eqn:CouplingMPrime} this proves the first assertion.

For the second assertion, by Lemma \ref{lem:RWCSPkilled}, the above coding processes have the same law under $\Pb(\cdot \mid h(\C) \geq n)$ and $\Psf$. The same arguments apply and conclude the proof.\end{proof}

\subsection{Proof of the $\normalfont{\IIC}$ results}\label{sec:IICTheorem}

\begin{proof}[Proof of Theorem \ref{thm:IICTheorem}]Let $R\in \Z_+$ and $\varepsilon>0$. We first prove that under $\Psf$, $\Pb(\cdot \mid h(\C) \geq n)$ and $\Pb_p(\cdot \mid \vert \C \vert = \infty )$, $M_{n,m}$ contains $\B_R(M_\infty)$ with high probability for a good choice of the parameters, where $M_\infty$ is the underlying infinite half-planar triangulation (with origin vertex $\rho$). This closely follows the proof of Theorem \ref{thm:DecompositionTheorem}, to which we refer for more details.

We first restrict our attention to $\Psf$, and define the random maps 
\[M_N:=M_{N,\xi_N}, \quad \xi_N:=\inf\left\lbrace A^{(l)}_k : k \geq T^{(A,l)}_N\right\rbrace,\quad N\in\N.\] We denote by $d_N$ the graph distance on $M_N$. By Proposition \ref{prp:Submap}, $M_N$ is $\Psf$-a.s.\ a sub-map of $M_\infty$ and we denote by $\partial M_{N}$ its boundary as such (as in \eqref{eqn:DefBoundary}). Recall that by Tanaka's theorem, we have $\xi_N\rightarrow \infty$ $\Psf$-a.s.\ as $N\rightarrow\infty$. Let $\{\tau_k : k\in\Z_+\}$, $\{\tau^{l}_k : k\in\Z_+\}$ and $\{\tau^{r}_k : k\in\Z_+\}$ be the endpoints of the excursion intervals of $A^{(r)}$, $V^{(l)}$ and $V^{(r)}$ above their infimum processes, as in (\ref{eqn:ExcursionIntervals}). They define cut-points that disconnect the origin from infinity in $\Lt_{A^{(l)},A^{(r)}}$, $\Lt_{V^{(l)}}$ and $\Lt_{V^{(r)}}$ (and thus in $\Hc$, $\Hl$ and $\Hr$) respectively, by the identities
\[v_k:=p_{A^{(r)}}\left( A^{(r)}_{\tau_k} \right), \ v^{l}_k:=p_{V^{(l)}}\left( V^{(l)}_{\tau^{l}_k} \right) \ \text{and} \ v^{r}_k:=p_{V^{(r)}}\left( V^{(r)}_{\tau^{r}_k} \right),\quad k\in\Z_+.\] The numbers of cut-points of $\Hc$ identified in $M_{N}$ read
\[K(N):=\# \left \lbrace k\in\Z_+ : \tau_k<T^{(A,r)}_{\xi_N} \right \rbrace,\] and similarly for $\Hl$ and $\Hr$ with 
\[K_l(N):=\# \left \lbrace k\in\Z_+ : \tau^{l}_k<T^{(V,l)}_{N} \right \rbrace \quad \text{and} \quad \ K_r(N):=\# \left \lbrace k\in\Z_+ : \tau^{r}_k<T^{(V,r)}_{\xi_N} \right \rbrace.\] Since the processes $A^{(r)}$, $V^{(l)}$ and $V^{(r)}$ are centered random walks and $\xi_N\rightarrow \infty$ $\Psf$-a.s., we get
\[K(N), \ K_l(N), \ K_r(N) \underset{N\rightarrow \infty}{\longrightarrow}  \infty \quad \Psf\text{-a.s.}.\] 
We define an equivalence relation $\approx$ as in Theorem \ref{thm:DecompositionTheorem}, by identifying vertices between consecutive cut-points in $\Hc$, $\Hc_l$ and $\Hc_r$. We denote the quotient map $M_N/\approx$ by $M'_N$ (the root edge of $M'_N$ is the root edge of $M_\infty$) and the graph distance on $M'_N$ by $d'_N$. The family $\{M'_{N} : n\in\N\}$ is a consistent sequence of locally finite maps with origin $v_0=\rho$. Moreover, for every $N\in\N$, the boundary of $M'_{N}$ in $M'_{N+1}$ is $\{v_{K(N)},v^{l}_{K_l(N)},v^{r}_{K_r(N)}\}$. Thus, the sequences \[\left\lbrace d'_N\left(\rho,v_{K(N)}\right) : N\in\N\right\rbrace, \ \left\lbrace d'_N\left(\rho,v^l_{K_l(N)}\right) : N\in\N \right\rbrace \ \text{and} \ \left\lbrace d'_N\left(\rho,v^r_{K_r(N)}\right) : N\in\N \right\rbrace \] are non-decreasing and diverge $\Psf$-a.s.. By definition of $\approx$, since we discover the finite regions that are swallowed during the exploration, the representatives of $\partial M_N$ in $M'_N$ are $v_{K(N)}$, $v^{l}_{K_l(N)}$ and $v^{r}_{K_r(N)}$. As a consequence,
\begin{equation}\label{eqn:DistanceBoundaryConvergenceIIC}
	d_N\left( \rho,\partial M_{N} \right)\underset{N \rightarrow \infty}{\longrightarrow}\infty \quad \Psf\text{-a.s.}.
\end{equation} Let us choose $N\in\N$ such that $\Psf(d_N(\rho,\partial M_{N})<R)\leq \varepsilon$. By the proof of Proposition \ref{prp:Submap}, we have that $\Psf$-a.s., for every $n\geq N$ and $m\geq \xi_N$, $M_N$ is a sub-map of $M_{n,m}$. Since $\xi_N$ is $\Psf$-a.s.\ finite, we can fix $m\in\N$ such that $\Psf(\xi_N>m)\leq \varepsilon$, and thus for every $n\geq N$,
\begin{equation}\label{eqn:SubmapFinite}
	\Psf(M_N\subseteq M_{n,m})\geq 1-\varepsilon
\end{equation} Since $\xi_n\rightarrow \infty$ $\Psf$-a.s., there exists $N_1\geq N$ such that for every $n\geq N_1$, $\Psf(\xi_n<m)\leq \varepsilon$. By Proposition \ref{prp:Submap}, it follows that for every $n\geq N_1$, 
\begin{equation}\label{eqn:Submap}
\Psf(M_{n,m}\subseteq M_\infty)\geq 1-\varepsilon.
\end{equation} By construction, when $M_N\subseteq M_{n,m}\subseteq M_\infty$, we have $d_N\left(\rho,\partial M_N\right)\leq d\left(\rho,\partial M_{n,m}\right)$ (where $d$ is the graph distance on $M_{n,m}$ and $\partial M_{n,m}$ its boundary in $M_\infty$). Thus, for every $n\geq N_1$,
\begin{equation}\label{eqn:DistanceRoot}
	\Psf(M_{n,m}\subseteq M_\infty, d\left(\rho,\partial M_{n,m}\right)\geq R) \geq 1-3\varepsilon.
\end{equation} By Proposition \ref{prp:Coupling}, we find $N_2\geq N_1$ such that for every $n\geq N_2$,
\begin{equation}\label{eqn:PropCoupling}
	\limsup_{p\downarrow p_c}{\sup_{A\in \mathcal{A}}{\left\vert \Pb_p(M_{n,m}\in A \mid \vert \C \vert =\infty)-\Psf(M_{n,m}\in A) \right\vert}}\leq \varepsilon,
\end{equation} while $M_{n,m}$ has the same distribution under $\Psf$ and $\Pb(\cdot \mid h(\C) \geq n)$. Under $\Pb(\cdot \mid h(\C) \geq n)$ and $\Pb_p(\cdot \mid \vert \C \vert = \infty )$, $M_{n,m}$ is a.s.\ a sub-map of $M_\infty$. By the construction of $M_{n,m}$, \eqref{eqn:DistanceRoot} and \eqref{eqn:PropCoupling} we have for every $n \geq N_2$,
\begin{equation}\label{eqn:DistanceRoot2}\limsup_{p\downarrow p_c}{\Pb_p(d\left(\rho,\partial M_{n,m}\right)<R \mid \vert \C \vert =\infty)}\leq 4\varepsilon \quad \text{and} \quad \Pb(d\left(\rho,\partial M_{n,m}\right)<R \mid h(\C) \geq n)\leq 3 \varepsilon.\end{equation} Finally, on the event $\{M_{n,m}\subseteq M_\infty\}$, $d(\rho,\partial M_{n,m})\geq R$ enforces that $\B_R(M_{n,m})=\B_R(M_\infty)$, which concludes the first part of the proof.

Now, let $A\in \mathcal{A}$ be a Borel set for the local topology. By \eqref{eqn:DistanceRoot2}, for every $n\geq N_2$,
\begin{align*}&\limsup_{p\downarrow p_c}{\left\vert \Pb_p(\B_R(M_{\infty})\in A \mid \vert \C \vert =\infty)-\Pb_p(\B_R(M_{n,m})\in A \mid \vert \C \vert =\infty) \right\vert}\\
&\leq 2\limsup_{p\downarrow p_c}{\Pb_p(d\left(\rho,\partial M_{n,m}\right)<R \mid \vert \C \vert =\infty)} \leq 8\varepsilon.
\end{align*}Then, by \eqref{eqn:PropCoupling}, for every $n\geq N_2$,
\[\limsup_{p\downarrow p_c}{\left\vert \Pb_p(\B_R(M_{n,m})\in A \mid \vert \C \vert =\infty)-\Psf(\B_R(M_{n,m})\in A) \right\vert } \leq \varepsilon.\] Finally, by \eqref{eqn:DistanceRoot}, for every $n\geq N_1$,
\[\left\vert \Psf(\B_R(M_{n,m})\in A)-\Psf(\B_R(M_{\infty})\in A) \right\vert\leq 2\left( 1- \Psf(M_{n,m}\subseteq M_\infty, d\left(\rho,\partial M_{n,m}\right)\geq R)\right)\leq 6\varepsilon.\] This concludes the proof under $\Pb_p(\cdot \mid \vert \C \vert = \infty )$ since $R$ and $\varepsilon$ are arbitrary. 

Under $\Pb(\cdot \mid h(\C) \geq n)$, the proof is simpler. By \eqref{eqn:DistanceRoot2} once again, for every $n\geq N_2$,
\begin{multline*}\left\vert \Pb(\B_R(M_{\infty})\in A \mid   h(\C) \geq n)-\Pb(\B_R(M_{n,m})\in A \mid  h(\C) \geq n) \right\vert\\
\leq 2\Pb(d_{\textrm{gr}}\left(\rho,\partial M_{n,m}\right)<R \mid h(\C) \geq n) \leq 6\varepsilon,
\end{multline*} and by Proposition \ref{prp:Coupling}, 
\[\Pb(\B_R(M_{n,m})\in A \mid  h(\C) \geq n)=\Psf(\B_R(M_{n,m})\in A).\] This concludes the proof. \end{proof}

\begin{Rk}In view of \cite[Theorem 1]{bertoin_conditioning_1994}, we believe that this proof can be adapted to show that in the sense of weak convergence, for the local topology \[\Pb_{p_c}(\cdot \mid T \geq n) \underset{n \rightarrow \infty}{\Longrightarrow}  \Piic.\] Otherwise said, the $\IIC$ arises as a local limit of a critically percolated $\UIHPT$ when conditioning the exploration process to survive a long time. It is also natural to conjecture that conditioning the open cluster of the origin to have large hull perimeter (as in \cite{curien_percolation_2014}) or to reach the boundary of a ball of large radius (as in \cite{kesten_incipient_1986}) yields the same local limit. However, our techniques do not seem to allow to tackle this problem.
\end{Rk}

\section{Scaling limits and perspectives}\label{sec:ScalingLimits}
In the recent work~\cite{baur_classification_2016} (see also \cite{gwynne_scaling_2016}), Baur, Miermont and Ray introduced the scaling limit of the quadrangular analogous of the $\UIHPT$, called the Uniform Infinite Half-Planar Quadrangulation ($\UIHPQ$). Precisely, they consider a map $\mathbf{Q}^\infty_{\infty}$ having the law of the $\UIHPQ$ as a metric space equipped with its graph distance $d_{\textrm{gr}}$, and multiply the distances by a scaling factor $\lambda$ that goes to zero, proving that~\cite[Theorem 3.6]{baur_classification_2016}
\[\left(\mathbf{Q}^\infty_{\infty},\lambda d_{\textrm{gr}} \right) \underset{\lambda \rightarrow 0}{\overset{(d)}{\longrightarrow}} \textsf{BHP},\] in the local Gromov-Hausdorff sense (see~\cite[Chapter 8]{burago_course_2001} for more on this topology). The limiting object is called the \textit{Brownian Half-Plane} and is a half-planar analog of the Brownian Plane of~\cite{curien_brownian_2013}. Such a convergence is believed to hold also in the triangular case.

We now discuss the conjectural continuous counterpart of Theorem \ref{thm:DecompositionTheorem}, and the connection with the \textsf{BHP}. On the one hand, the processes $B$ and $W$ introduced in Section \ref{sec:PercolationClusters} have a scaling limit. Namely, using the asymptotic (\ref{asymptZ}) and standard results of convergence of random walks with steps in the domain of attraction of a stable law~\cite[Chapter 8]{bingham_regular_1989} one has
\[\left(\lambda^{2}B_{\lfloor t/\lambda^3 \rfloor}\right)_{t\geq 0} \underset{\lambda \rightarrow 0}{\overset{(d)}{\longrightarrow}} (X_t)_{t \geq 0},\] in distribution for Skorokhod's topology, where $X$ is (up to a multiplicative constant) the spectrally negative $3/2$-stable process. This suggests that the looptrees $\Lt_B$ and $\Lt_W$  converge when distances are rescaled by the factor $\lambda^2$ towards a non-compact version of the random stable looptrees of~\cite{curien_random_2014} (with parameter $3/2$), in the local Gromov-Hausdorff sense. This object is supposed to be coded by the process $X$ extended to $\mathbb{R}$ by the relation $X_t=-t$ for every $t\leq 0$ and equivalence relation
\begin{equation}\label{eqn:EquivX}
  s \sim t \quad \text{iff} \quad X_t=X_s=\inf_{s \wedge t \leq u \leq s \vee t}{X_u}.
\end{equation} On the other hand, one can associate to each negative jump of size $\ell$ of $X$ (which codes a loop of the same length in the infinite stable looptree, see~\cite{curien_random_2014}) a sequence of jumps of $B$ (equivalently, of loops of $\Lt_B$) with sizes $\{\ell_\lambda : \lambda>0\}$ satisfying  
\[\lambda^2\ell_\lambda \underset{\lambda \rightarrow 0}{\longrightarrow} \ell.\] With each negative jump of $B$ is associated a Boltzmann triangulation $M_{\lambda}$ with a simple boundary of size $\ell_\lambda$ and graph distance $d_\lambda$, that fills in the corresponding loop in the decomposition of the $\UIHPT$. Inspired by~\cite[Theorem 8]{bettinelli_compact_2015}, we expect that there exists a constant $c>0$ such that 
\[\left(M_{\lambda},(c\ell_\lambda)^{-1/2}d_\lambda \right) \underset{\lambda \rightarrow 0}{\overset{(d)}{\longrightarrow}} \textsf{FBD}_1,\] in the Gromov-Hausdorff sense, where $\textsf{FBD}_1$ is a compact metric space called the Free Brownian Disk with perimeter 1, originally introduced in~\cite{bettinelli_scaling_2015} (this result has been proved for Boltzmann bipartite maps with a general boundary). By a scaling argument, we obtain 
\[\left(M_{\lambda},\lambda d_\lambda \right) \underset{\lambda \rightarrow 0}{\overset{(d)}{\longrightarrow}} \textsf{FBD}_{c\ell},\] where $\textsf{FBD}_{c\ell}$ is the \textsf{FBD} with perimeter $c\ell$. From these observations, it is natural to conjecture that the looptrees  $\Lt_B$ and $\Lt_W$ filled in with independent Boltzmann triangulations converge when rescaled by a factor that goes to zero, towards a collection of independent \textsf{FBD} with perimeters given by the jumps of a Lévy $3/2$-stable process, and glued together according to the looptree structure induced by this process. Based on Theorem \ref{thm:DecompositionTheorem}, we believe that there is a way to glue two independent copies of the above looptrees of brownian disks along their boundaries so that the resulting object has the law of the \textsf{BHP}. Similarly, one may give a rigorous sense to the $\IIC$ embedded in the \textsf{BHP}, by passing to the scaling limit in Theorem \ref{thm:IICTheorem}. However, there are several possible metrics on the topological quotient, and it is not clear which metric should be chosen. 

\smallskip

This question is connected to~\cite{duplantier_liouville_2014,sheffield_conformal_2010}, where gluings of quantum surfaces are discussed. As we already mentioned, Theorem \ref{thm:DecompositionTheorem} can be seen as a discrete counterpart to~\cite[Theorem 1.16-1.17]{duplantier_liouville_2014}. Applying these results with the choice of parameter $\gamma=\sqrt{8/3}$, we obtain that the $\theta=\pi$ quantum wedge (which is believed to be a conformal version of the \textsf{BHP}) is obtain as the gluing of two independent forested wedges with parameter $\alpha=3/2$ (an infinite counterpart to the $3/2$-stable looptree, where loops are filled in with disks equipped with a metric defined in terms of the Gaussian Free Field - a conformal version of the brownian disk). Moreover, the counterflow line that separates the forested wedges is a Schramm-Loewner Evolution with parameter 6, which is also the scaling limit of percolation interfaces in the triangular lattice~\cite{schramm_scaling_2000,smirnov_critical_2001}. In the recent work \cite{gwynne_convergence_2017}, Gwynne and Miller also proved the convergence of the percolation interface for face percolation on the $\UIHPQ$ towards Schramm-Loewner Evolution with parameter 6 on the $\textsf{BHP}$.

\begin{Ack}I would like to thank Gourab Ray for suggesting this question and interesting discussions, Grégory Miermont for insightful comments and many useful suggestions and Erich Baur for his helpful remarks. I am greatly indebted to Nicolas Curien for pointing out an inaccuracy on a preliminary version of this work.\end{Ack}

\bibliography{IICUIHPT.bib}
\bibliographystyle{abbrv}

\end{document}